\documentclass[10pt]{article}
\usepackage{amsmath}
\usepackage{amssymb}
\usepackage{indentfirst}
\usepackage[latin1]{inputenc}
\usepackage{geometry}
\usepackage{amsfonts}
\usepackage{graphicx}
\usepackage{subfigure}
\usepackage{float}
\usepackage{color}
\usepackage{verbatim}
\usepackage{enumerate}

\newcommand{\N}{{\mathbb N}}
\newcommand{\Z}{{\mathbb Z}}

\newcommand{\Q}{{\mathbb Q}}
\newcommand{\R}{{\mathbb R}}
\newcommand{\C}{{\mathbb C}}

\newcommand{\F}{{\mathcal F}}

\newcommand{\HQ}{{\mathbb H}}
\newcommand{\BQ}{{\mathbb B}}
\newcommand{\h}{{\mathcal{H}}}   
\newcommand{\SB}{{\rm SB}}   

\newcommand{\vp}{\vspace{0.3cm}}

\newcommand{\bc}{\begin{center}}
\newcommand{\ec}{\end{center}}

\newcommand{\Fe}{{\cal F}}

\newcommand{\U}{{\cal U}}

\newcommand{\GL}{{\rm GL}}
\newcommand{\SL}{{\rm SL}}
\newcommand{\PSL}{{\rm PSL}}
\newcommand{\SU}{{\rm SU}}

\newcommand{\Iso}{{\rm ISO}}
\newcommand{\tr}{{\rm tr}}

\renewcommand{\O}{\mathcal{O}}

\newcommand{\qed}{\enspace\vrule  height6pt  width4pt  depth2pt}
\newenvironment{proof}{\par\noindent{\bf Proof.}}{$\qed$\par\bigskip}

\newtheorem{theorem}{Theorem}[section]

\newtheorem{lemma}[theorem]{Lemma}

\newtheorem{proposition}[theorem]{Proposition}
\newtheorem{remark}[theorem]{Remark}

\begin{document}

 \title{From the Poincar\'e Theorem to generators of the unit group of integral group rings of finite groups\thanks{Mathematics subject Classification Primary [16S34, 16U60]; Secondary [20C05].
Keywords and phrases: Units, Group Ring, Fundamental Domain, Generators.
\newline The first author is supported in part by Onderzoeksraad of
Vrije Universiteit Brussel and Fonds voor
Wetenschappelijk Onderzoek (Flanders). The second author was partially supported by CNPq and FAPESP-Brazil, while visiting the Vrije Universiteit Brussel.  The third  author is supported by Fonds voor
Wetenschappelijk Onderzoek (Flanders)-Belgium.  The fourth author  is supported by FAPESP and CNPq-Brazil.  The last author is supported by FAPESP
(Funda\c c\~ao  de Amparo \`a
Pesquisa do Estado de S\~ao Paulo), Proc. 2008/57930-1.
 }}

\author{E. Jespers \and S. O. Juriaans \and  A. Kiefer \and A. de A. e Silva \and A. C. Souza Filho}
\date{}

\maketitle

\begin{abstract}
We give an algorithm to determine finitely many generators for a subgroup of finite index in the unit group of an integral group ring $\Z G$
of a finite nilpotent group $G$, this provided the rational group algebra $\Q G$ does not have simple components that are division classical quaternion algebras or two-by-two matrices over a classical quaternion algebra with centre $\Q$. The main difficulty is to deal with orders in
quaternion algebras over the rationals or a quadratic imaginary extension of the rationals. In order to deal with these we
give a finite and easy implementable algorithm to compute a fundamental domain in the hyperbolic three space  $\HQ^3$ (respectively hyperbolic two space $\HQ^2$) for a discrete subgroup of $\PSL_2(\C)$ (respectively $\PSL_2(\R)$) of finite covolume.
Our results on group rings are a continuation of earlier work of Ritter and Sehgal,  Jespers and Leal.
 \end{abstract}


\section{Introduction}

The unit group of an order in a finite dimensional semisimple rational algebra $A$  is an
important example of an arithmetic group. Hence it forms a fundamental topic of interest. Recall
that a subring $\O$  of $A$ is said to be an order if $\O$  is a finitely generated $\Z$-module that contains a
$\Q$-basis of $A$. Prominent examples of orders are group rings $RG$ of finite groups $G$ over the ring of
integers $\O_{K}$ of an algebraic number field $K$. The unit group $\U (RG)$ of $RG$ has received a lot of attention
and most of it has been given to the case $R = \Z$; for surveys we refer to \cite{klei-sur,sehmil,seh2,seh3}. It is well known
that the unit group $\U (\O)$  of an order $\O$ is a finitely presented group. However, only for very few
finite non abelian groups $G$ the unit group of $\Z G$ has been described, and even for fewer groups $G$ a
presentation of $\U (\Z G)$ has been obtained. Nevertheless, for many finite groups $G$ a specific finite set
$B$ of generators of a subgroup of finite index in $\Z G$ has been given.

Bass and Milnor in \cite{bas} showed that if $A$ is a finite abelian group then the so called Bass units generate a subgroup
of finite index in $\U (\Z A)$. Recall that a Bass unit in the integral group ring $\Z G$ of a finite group $G$ is a unit of
the form $(1+ g + \cdots + g^{i-1})^{m} + \left( \frac{1-i^{m}}{o(g)} \right)  \widehat{g}$,
where $g \in G,$ $1< i < o(g)$,  $(i,o(g))=1$ and
$m$ is the order of $\ i$ in $\U (\Z_{o (g)})$ (or
one takes a fixed multiple of $\ m$, for example $\varphi
(|G|)$), where  $\varphi$ is the Euler
$\varphi$-function and $\hat{g}=1+g+\cdots +g^{o(g)-1}$.

This result stimulated the search for new units that generate a subgroup of finite index in the unit group of the integral group ring
of a non-commutative finite group $G$.
Finding finitely many generators for $\U (\Z G)$  is one of the important problems in the field (see Problems 17 and 23 in \cite{seh2}).
For this purpose Ritter and Sehgal introduced the so called bicyclic units, these are the
unipotent units of the form $ 1+(1-g)h\widehat{g}$ and $1+\widehat{g}h (1-g) $, with $g,h \in G$.  In a series of papers, Ritter and Sehgal (see for example \cite{rit-seh2,rit-seh1,rit-seh3})
showed that for several classes of finite groups, including nilpotent groups of odd order, the group generated by both the Bass units and the bicyclic units
is of finite index in $\U (\Z G)$. A remarkable result as one only knows two types of generic units.

Jespers and Leal in \cite{jes-leal-manus} extended this result to a much wider class of finite  groups, only excluding  those finite groups $G$ for which
the Wedderburn decomposition of the rational group algebra $\Q G$ has certain simple components of degree one or two over a division algebra
and groups which have a non-abelian fixed point free epimorphic image.
In order to state the precise restrictions and also to clarify the reason for these restrictions, it is convenient to work in the more general
context of orders.

So, let $A$ be a semisimple finite dimensional rational algebra. If $e_{1}, \ldots , e_{n}$ denote the primitive
central idempotents of $A$ then $Ae_{i}=M_{n_i}(D_{i})$, with  $D_{i}$ a skew field and $n_i\in \mathbb{N}$, are the Wedderburn components
of $\ A$. If $\O_{i}$ is an order
in $D_{i}$ then $\prod\limits_{i}M_{n_i}(\O_{i})$ is an order   in $A$. Hence if $\U(\O )$ is the group of
units of an order $\O$ in $A$, then $\U (\O)$ contains a subgroup of finite index of the form $V_1 \times \cdots \times V_n $,
with each  $V_i$  a subgroup of finite index in $\GL_{n_i}(\O_{i}) = \U (M_{n_i}(\O_{i}))$.
Let  $\SL_{n_i}(\O_{i})$  denote  the group of matrices of reduced norm $1$.  For an ideal $\cal{Q}$ of $\O_{i}$,   let
 $E_{n_{i}}({\cal Q} )=\langle  I+qE_{lm}, q\in
{\cal Q}, \, 1\leq l,m\leq n_{i},\, l\neq m  \rangle \leq  \SL_{n_i}(\O_{i})$. The  celebrated
theorems of Bass-Milnor-Serre
\cite{bass-milnor-serre}, Liehl  \cite{Liehl1981}, Vaserstein
\cite{vaserstein}, Bak-Rehmann
\cite{bak-rehman} and   Venkataramana \cite{Venkataramana1994} state that if $n_{i}\geq 3$ or $n_{i}=2$ and $D_{i}$ is different from $\Q$,
a quadratic imaginary extension  of $\Q$ and a totally definite quaternion algebra with center $\Q$ then
$[\SL_{n_{i}}({\cal O}_{i}): E_{n_{i}}({\cal Q})] < \infty$ for any non-zero ideal ${\cal Q}$ of $\O_{i}$. For more details and background
we refer to \cite{rehman,seh2}. Recall that  $D_{i}$ is a totally definite  quaternion algebra if and only if $D_{i}$ is not
commutative and $\U(Z({\cal
O}_{i}))$ has finite index in $\U (\O_{i})$ (see \cite{KleinertTotDef,seh2}).
Because of these results we call an exceptional component of $A$ an epimorphic image of $A$ that is either a non-commutative division algebra other than a
totally definite quaternion algebra, or a two-by-two matrix ring over the rationals, a quadratic imaginary extension of the rationals or a
quaternion algebra
$\h (a,b,\Q )$ with $a$ and $b$ negative integers (see section 2 for the notation).

The proof of the  main result in \cite{jes-leal-manus} shows that if $G$ is a finite group so that $\Q G$ does not have exceptional simple components and
$G$ has non-commutative
fixed point free images then the group generated by the Bass and bicyclic units is of finite index in $\U (\Z G)$. Note that in \cite{jes-leal-manus}
one also excluded simple components of $\Q G$ of the type $M_{2}(D)$, with $D$ an arbitrary non-commutative division algebra; however because of the
 results in \cite{Venkataramana1994} one only has to exclude division algebras of the type $\h (a,b,\Q)$. It follows from the proof that
the bicyclic units generate a group so that it contains a subgroup of finite index in $1-e_{i}+\SL_{n_{i}}(\O_{i})$
for every simple component $\Q G e_{i}=M_{n_{i}}(D_{i})$
of $\Q G$ with $n_{i}\geq 2$ and $\O_{i}$ an order in the   non-commutative division algebra $D_{i}$. To show this one proves that, because $Ge_{i}$ is not fixed point free by assumption, there exists $g\in G$ so that
$f_{i}=\frac{1}{o(g)} \hat{g}e_{i}$ is a non-central idempotent of $\Q Ge_{i}$ and then one shows that the group generated by bicyclic units built on $g$ contains all matrices of the
type $1-e_{i} +q E_{kl}$ for $q$ in  a nonzero ideal of $\O_{i}$ and $1\leq k,l\leq n_{i}$ with $k\neq l$. It then follows from the above mentioned
theorem that this generates a subgroup of finite index in $\SL_{n_{i}}(\O_{i})$ (crucial is to have described  a non-
central idempotent $f_{i}\in \Q Ge_{i}$ and consider all the generators of the type $1+n_{f_{i}}^{2}(1-f_{i}) h f_{i},
\; 1+n_{f_{i}}^{2}f_{i}h(1-f_{i})$ where $h\in G$ and $n_{f_{i}}$ is a positive integer so that $n_{f_{i}}f_{i}\in \Z 
G$). We identify $\O_{i}$ with the scalar matrices with entries in $\O_{i}$. It is well known that  $\U 
(\mathcal{Z} (O_{i})) \, \SL_{n_{i}}(\O_{i}) $ is of finite index in $\GL_{n_{i}} (\O_{i} )$.
The Bass cyclic units (together with the bicylic units) are then ``used'' to
generate a subgroup of finite index in $1-e_{i} +\U (\mathcal{Z} (\O_{i}))$ (note that Bass units are not necessarily central). The proof for this  makes use of a result of Bass \cite{bas} on $K_{1}(\Z G)$. Of course, ultimately this makes use of the Dirichlet Unit theorem on the description of the unit group of the ring of integers $\O_{K}$ in a number field $K$. In \cite{jes-par-seh} the use of the Bass units is made
very clear in case $G$ is a finite nilpotent group. Indeed, in this case, it is shown that the group generated by the Bass units contains a subgroup of finite index in $\mathcal{Z} (\U (\Z G))$, i.e. it contains a subgroup of finite index in $1-e_{i}+e_{i}\U (\mathcal{Z} (\O_{i}))$ for every $i$.
To state the precise result, we recall some notation. We denote by $Z_{i}$ the $i$-th centre of $G$. For  $g\in G$ and a Bass
unit $b\in \Z \langle g\rangle $ put $b_{(1)}=b$, and, for $2\leq i\leq n$, put $b_{(i)}=\prod_{x\in Z_{i}} b_{(i-1)}^{x}$, where
$\alpha^{x}=x^{-1}\alpha x$ for $\alpha \in \Z G$. If $n$ is the nilpotency class of $G$, then the group $\langle b_{(n)}\ |\ b \ \mbox{a Bass cyclic unit}\rangle$ is of finite index in $\mathcal{Z} (\U (\Z G))$.

If one considers group rings $RG$ of finite groups over say a commutative order $R$ in a number field $K$ that is larger than $\Z$, then often
the group algebra $KG$ has less exceptional components (or none at all if for example $K=\Z [\xi ]$ with $\xi$ a primitive root of unity
of order $G$) and one can also apply the method explained to obtain that the group generated by the (generalized) Bass and bicyclic units
generates a subgroup of finite index in $\U (RG)$. We refer the reader to \cite{jes-leal-manus} for examples and more details.
Note that in  \cite{jes-leal-deg12} a  method is described to obtain generators for a subgroup of finite index in
$\U (\Z G)$ for an arbitrary nilpotent finite group. These generators are determined  in function of knowing generators of unit groups $\U (\Z H)$ of
some special epimorphic images $H$  of $G$ that determine exceptional simple components of $\Q G$. However, ultimately one still has to
deal with unit groups of orders in exceptional simple components.

So the remaining problem for describing a finite set of generators, up to finite index,
of the unit group of $\Z G$ of an arbitrary finite group $G$, is  dealing with orders in exceptional components and with groups $G$ that
are non-abelian and fixed point free. The latter groups are precisely  the Frobenius complements  (see \cite[Theorem 18.1.iv]{passman-perm}).
In particular, the non-abelian nilpotent finite fixed point free groups are
$F=Q_{2^{m}}\times C_{n}$, with  $Q_{2^{m}}$  a generalized quaternion group of order $2^{m}\geq 8$ and  $C_{n}$  a cyclic group of odd order $n$.
Suppose now that $G$ is a nilpotent finite group so that $Ge=F$.
By $\xi_{k}$ we denote a primitive root of unity of order $k$.
If $n=1$ then
$\Q G e =\h (-1,-1, \Q(\xi_{2^{m-1}}+\xi_{2^{m-1}}^{-1}))$, a totally definite quaternion algebra, and hence the central units of $\Z G$ contain a subgroup of finite index in the unit group of an order in this ring. If   $n$ is an odd positive integer so that the multiplicative order of $2$ modulo $n$ is odd, then $\Q Fe =\h (-1,-1, \Q(\xi_{2^{m-1}}+\xi_{2{m-1}}^{-1}, \xi_{n}) ) $ is a division algebra, otherwise it is a matrix algebra that is only exceptional if $m=3$ and $n=3$. In the latter case one can describe non-central idempotents and hence units that generate a subgroup of finite index in an order of the algebra.
For details on this we refer to \cite{jes-olt-rio} and the last section.
A minimal example of the former case is $G=Q_{8}\times C_{7}$. Here  $\Q G$ has only one non-commutative simple component, the division algebra
$\h (-1,-1,\Q(\xi_{7}))$. So far no methods have been discovered in order to deal with such exceptional components, i.e. to deal with components that
are non-commutative division algebras that are not totally definite quaternion algebras (see \cite{kle-boo}). However, a first step  was made in \cite{jesetall}, where an  algorithm  is given to compute the unit group of the  order $\h (-1,-1,\O_{K})$  in a division  algebra $\h (-1,-1, K)$ over an imaginary quadratic field extension $K$ of the rationals.  The algorithm computes a fundamental polyhedron for the cocompact discrete group $\U(\h (-1,-1, \O_{K}))$ that acts on the hyperbolic three space (and thus determines generators by making use of the Poincar\'{e} theorem). Next the algorithm is  applied to the unit group of $\h (-1,-1, \Z [\frac{1+\sqrt{-7}}{2}])$  providing  the first example of a description of  a finite set of generators of the unit group of a  group ring $RG$ with an exceptional component that is a division algebra which is not a totally definite quaternion algebra.  Note that $\h (-1,-1,\Q(\sqrt{-7})) \subseteq \h (-1,-1,\Q (\xi_{7}))$. Subsequently, in \cite{jpsf}, an explicit method was given to construct units in $\h(\mathbb{Z}[\frac{1+\sqrt{-d}}{2}])$ for all positive $d\equiv 7\ \mbox{mod}\ 8$ and it was shown that a unit group of any order of such algebras is a hyperbolic group.

In \cite{dooms-jespers-konovalov}, Dooms, Jespers and Konovalov introduced a method (also by computing a fundamental polyhedron of a discrete group
of finite covolume) to deal with exceptional simple components of the type
$M_2(\Q)$. New generators are introduced, using Farey
symbols, which are in one to one correspondence with fundamental polygons
of congruence subgroups of $PSL_2(\Z)$.

In \cite{jeskiefjuretall} Jespers, Juriaans, Kiefer,  de Andrade e Silva  and Souza Filho,  obtained explicit formulas for the bisectors in the Poincar\'e Theory of hyperbolic $2$ and $3$ space. These formulas, and their consequences, were used to revisit the work of  Swan on Bianchi groups, give an easy verifiable criteria for non-compactness, a geometric criteria to determine the type  of a M�bius transformation and will here be used  to revisit the works of \cite{jesetall, katok, riley}.

In this paper we give a method that deals with all exceptional simple components that are matrix algebras over fields and we also can handle division algebras that are quaternion algebras with centre either the rationals or a quadratic imaginary extension of the rationals.
We do this via actions on hyperbolic spaces. Describing generators and relations of groups acting on hyperbolic spaces was started in
the nineteenth century. The main difficulty one encounters is the construction of a fundamental domain. This problem was considered by  Ford, Poincar\'{e}, Serre, Swan, Thurston and many others. Only in the case of a Ford domain explicit formulas are known. Computer aided methods also exist.  For   Fuchsian groups we refer to  \cite{johansson,katok}, for  Bianchi  groups we refer to  \cite{riley} and for  cocompact groups   we refer to \cite{jesetall}.
Our method is based on the authors' work \cite{jeskiefjuretall}, where   explicit descriptions of the bisectors in the Poincar\'e Theory (see \cite{ratcliffe}) in $\HQ^n$, with $n\in \{2,3\}$, are given. For convenience and completeness' sake, we recall in Section 2 the necessary background on the Poincar\'{e} theorem and in the first part of Section 3 we recall some formulas and results proved in \cite{jeskiefjuretall}. Since the latter mainly deals with the geometry of fundamental domains and to make the present paper independent and complete, we will reprove these results. Equipped with these tools, we turn to describing generators, up to finite index, of  a discrete group $\Gamma$ acting on a hyperbolic $2$ or $3$-space  and having finite covolume (coarea).  In Section 3 we describe our algorithm, called the Dirichlet  Algorithm of Finite Covolume (DAFC for short), to obtain generators for a subgroup of finite index of $\Gamma$. The algorithm DAFC is very easy to implement and reduces  much the computational time of earlier known examples.  In Section 4 we give several examples.
First we treat examples coming from division algebras and hence ameliorate \cite{jesetall}. Next we apply our method to matrix algebras and hence show that our method
does not require the condition of cocompactness, which was required in \cite{jesetall}. In Section 5 we then give applications to group rings and we focus on integral
group rings of finite nilpotent groups although applications for arbitrary finite groups also can be given. 


\section{Preliminaries}
\label{secback}

In this section we introduce the notation and recall some basic facts on hyperbolic spaces and quaternion algebras. Standard references on hyperbolic geometry are \cite{beardon, bridson,
elstrodt, gromov, mac, ratcliffe}. Let $\HQ^n$ (respectively $\BQ^n$) denote the upper half space (respectively the ball) model of hyperbolic $n$-space. So $\HQ^{3}= \C \times ]0, \infty [$ and we shall often think of $\ \HQ^{3}$ as a subset of the classical real quaternion algebra $\h=\h(-1,-1,{\R})$
by identifying  $\HQ^{3}$ with the subset  $  \{z+rj
\in \h : z \in \C, r \in \R^{+}\}\subseteq \h$. The
ball model $\BQ^{3}$ may be identified in the same way  with
$\{z+rj\in \C +\R j\ | \  |z|^2+r^2<1\}\subseteq
\h$.  Denote by $\mbox{Iso} (\HQ^{3})$ the group of isometries of $\ \HQ^{3}$. The group of
orientation preserving isometries is denoted by $\mbox{Iso}^{+}(\HQ^{3})$.
It is well known (see for instance \cite{elstrodt}) that $\mbox{Iso}^{+}(\HQ^{3})$ is isomorphic with
$\PSL_2(\C)$ and  $\mbox{Iso}(\HQ^{3})$ is isomorphic with
$\PSL_2(\C)\times C_2$, where by $C_n$ we denote the cyclic group of order $n$. More concretely, the action of $\ \PSL_2(\C)$  on  $\HQ^{3}$  is given
by
\begin{eqnarray*}
\begin{pmatrix}
a & b\\
c & d
\end{pmatrix} (P) = (aP+b)(cP+d)^{-1},
\end{eqnarray*}
where $(aP+b)(cP+d)^{-1}$  is evaluated in the algebra $\h$. Explicitly, if $P=z+rj$ and $\gamma = \begin{pmatrix}
a & b\\
c & d
\end{pmatrix}$ then $(aP+b)(cP+d)^{-1} = \frac{(aP+b)(\overline{P}\overline{c}+\overline{d})}{|cz+d|^{2}+|c|^{2}r^{2}}$, where $\overline{}$ denotes the classical involution on the algebra $\h$. Now the numerator may be written as $(az+b)(\overline{c}\overline{z}+\overline{d})+a\overline{c}r^{2}+arj(\overline{c}\overline{z}+\overline{d})-(az+b)rj\overline{c}$, where the last two terms may be written as $a(cz+d)rj-(az+b)crj$. Thus the $k$-component disappears and we obtain
\begin{eqnarray*}
\gamma (P)=\frac{(az+b)(\overline{c}\overline{z}+\overline{d})+a
\overline{c}r^{2}}{|cz+d|^{2}+|c|^{2}r^{2}} + (\frac{r}{|cz+d|^{2}+|c|^{2}r^{2}})j.
\end{eqnarray*}
This action may be extended to $\widehat{\HQ}^3=\HQ^3 \cup \partial \HQ^3 \cup \lbrace \infty \rbrace$ in the following way: if $P \in \partial \HQ^3$ and $P \neq -\frac{d}{c}$, then $\begin{pmatrix} a & b \\ c & d \end{pmatrix}(P)=(aP+b)(cP+d)^{-1}$, where the latter is simply evaluated in $\C$, if $P=-\frac{d}{c}$, then $\begin{pmatrix} a & b \\ c & d \end{pmatrix}(P)=\infty$ and $\begin{pmatrix} a & b \\ c & d \end{pmatrix}(\infty)=\frac{a}{c}$.
\begin{remark}
Note that throughout the article we write $a=a(\gamma )$, $b=b(\gamma )$, $c=c(\gamma )$ and $d=d(\gamma )$, for the entries of $\gamma =\begin{pmatrix}
a & b\\
c & d
\end{pmatrix}\in M_2(\C)$, when it is necessary to stress the dependence of the entries on the matrix $\gamma$.
\end{remark}

We will now analyse the orientation preserving isometries of the ball model $\BQ^3$. Therefore let $u=u_0+u_1i+u_2j+u_3k \in \h$ and define $\overline{u}$ to be $u_0-u_1i-u_2j-u_3k$, the conjugate of $u$. Moreover let $u'=u_0-u_1i-u_2j+u_3k$ and $u^{*}=u_0+u_1i+u_2j-u_3k$.
Define
\begin{equation}
\SB_{2} (\h )=\left\{
\left(
\begin{array}{cc}
a & b \\ c & d
\end{array}
\right) \in M_2(\h) |\  d=a^{\prime},\;  b=c^{\prime},\;  a\overline{a}-c\overline{c}=1\right\}.
\end{equation}
Note that if $f= \left(
\begin{array}{rr}
a & c' \\ c & a'
\end{array}
\right) \in \SB_{2} (\h )$ then $f^{-1}=  \begin{pmatrix}
\overline{a} & -\overline{c}\\
-c^* & a^*
\end{pmatrix}$. The following proposition gives the relation between the upper half space model $\HQ^3$ and the ball model $\BQ^3$ and gives the group of orientation preserving isometries of the latter space, $\mbox{Iso}^{+}
(\BQ^3 )$.

\begin{proposition}\label{propels}\cite[Proposition I.2.3]{elstrodt}
\begin{enumerate}[(i)]
\item For $P \in \HQ^3$, the quaternion $-jP+1$ is invertible in $\h$ and the  map $\eta_0 :
\HQ^3\longrightarrow \BQ^3$, given by
$ \eta_0(P)=(P-j)(-jP+1)^{-1}$, is an isometry. More precisely $\eta_0=\mu\pi$, where $\pi$ is the reflection in the border of $\HQ^3$ and $\mu$ is the reflection in the Euclidean sphere with centre $j$ and radius $\sqrt{2}$.
\item Let $g=\frac{1}{\sqrt{2}}\left(
\begin{array}{ll}
1 & j \\ j & 1
\end{array}
\right) \in M_{2} (\h )$. The map
 $\Psi : \SL_2(\C) \rightarrow SB_{2}(\h )$ given by  $\Psi ( \gamma ) =
\overline{g}\gamma g $
is a group   isomorphism.
\item For $u \in \BQ^3$ and $f=\begin{pmatrix} a & c' \\ c & a' \end{pmatrix} \in SB_2(\h)$ the quaternion $cu+a'$ is invertible in $\h$ and the transformations $f:\BQ^3 \longrightarrow \BQ^3$, defined by $f(u)=(au+c')(cu+a')^{-1}$ are isometries of $\BQ^3$ and define an action of $SB_2(\h)$ on $\BQ^3$. Again this action may be extended to the closure of $\BQ^3$, which we denote by $\overline{\BQ^3}$.
\item The group $\mbox{Iso}^{+}(\BQ^3 )$ is isomorphic with $SB_2(\h)/ \lbrace 1,-1 \rbrace$.
\item The map $\eta_0$ is equivariant with respect to $\Psi$, that is $\eta_{0} (MP)=\psi (M) \eta_{0}(P)$, for $P\in \HQ^{3}$ and $M\in \SL_2(\C)$ .
\end{enumerate}
\end{proposition}

Note that item (i) clearly shows that the map $\eta_0$ is a M\"obius transformation. Also note that an explicit formula for  Proposition \ref{propels} (ii) is
\begin{equation}\label{psigamma}
\Psi (\gamma)= \frac{1}{2}
\begin{pmatrix}
a +\overline{d} + (b-\overline{c})j&
b+\overline{c}+ (a-\overline{d})j \\
c+\overline{b}+(d-\overline{a})j & \overline{a}+d
+ (c-\overline{b})j ,
\end{pmatrix}
\end{equation} for $\gamma=\begin{pmatrix} a & b \\ c & d \end{pmatrix} \in \SL_2(\C)$. Hence $\Vert \Psi(\gamma) \Vert^2 = \Vert \gamma \Vert^2$, where the matrix norm $\Vert \gamma \Vert^2$ is defined as $\vert a \vert^2+\vert b \vert^2+\vert c \vert^2+\vert d \vert^2$.


The theory above shows how discrete subgroups of $\SL_2(\C)$ act discontinuously on $\HQ^3$ (respectively $\HQ^2)$.
Recall from the introduction that, for our purposes, we mainly are interested in such discrete  subgroups  that are determined by orders in the exceptional components
of rational group algebras $\Q G$, and the aim is to find explicit generators for them. We now make the link between quaternion algebras over number fields $K$ and a discontinuous action on hyperbolic space.
Let $K$ be an algebraic  number field and let
$\mathfrak o_K$ be an order in $K$. For $a$ and $b$ nonzero elements of $\ K$, we denote by
$\h(a,b,K)$  the generalized
quaternion algebra over $K$, that is
 $\h(a,b, K )=K [i,\,j:i^{2}=a,\,j^{2}=b,\,ji=-ij]$. In the particular case in which $a=b=-1$, we simply denote this algebra as $\h (K)$ and as $\h$ if furthermore $K=\R$. The set $\{1,\,i,\,j,\,k=i j \}$ is
an additive $K$-basis of $\ \h(a,b,K)$.
If $a,b\in \mathfrak{o}_{K}$ then we denote by $\h (a,\,b, \mathfrak{o}_{K})={\mathfrak{o}_K }
 +{\mathfrak{o}_{K}}i+ {\mathfrak{o}_{K}}j+{\mathfrak{o}_{K}}k$, a subring of $\h(a,b,K)$.   By $N$ we denote the usual norm on $\h (a,b, K)$, that is $N(u_{0}+u_{1}i +u_{2}j+u_{3}k) =u_{0}^{2}-au_{1}^{2}-bu_{2}^{2}+abu_{3}^{2} $. In the special case of $\h$ we simply denote $N(x)$ as $\vert x \vert^2$.
Denote by  $\SL_{1}(\h(a,\,b, \mathfrak{o}_K))$  the multiplicative group $\{x \in \h(a,b,\mathfrak {o}_K)\ |\ N(x)=1\}$. It is well known that the unit group $\U(\h(a,b,\mathfrak{o}_k))$ of $\h(a,b,\mathfrak{o}_k)$ is commensurable (i.e. has a common subgroup of finite index) with $\U(\mathfrak{o}_k)\cdot \SL_1(\h(a,b,\mathfrak{o}_k))$. Since the Dirichlet unit theorem deals with the structure of $\U(\mathfrak{o}_k)$, and as explained in the introduction the central units of integral group rings are "under control", we need to investigate $\SL_1(\h(a,b,\mathfrak{o}_k))$.

Similar as in the case of a classical quaternion algebra, given $u=u_0+u_1i+u_2j+u_3k\in \h(a,b,K)$, let $\overline{u}=
u_0-u_1i-u_2j-u_3k$,
$u^{\prime}= k^{-1}u k$ and let $u^{*}=
u_0+u_1i+u_2j-u_3k,$. The mapping $u\mapsto u'$ defines an
algebra isomorphism of $\h(a,b,K)$ and both $u\mapsto u^{*}$ and $ \ u\mapsto \overline{u}$ define  involutions of $\ \h(a,b,K)$.  The map
\begin{eqnarray*}
\h(a,b,K) \rightarrow M_2(\C)
\end{eqnarray*}
given by
\begin{eqnarray*}
u=u_0+u_1i+u_2j+u_3k\mapsto \gamma_u= \begin{pmatrix}
u_0+u_1\sqrt{a} & u_2\sqrt{b}+u_3\sqrt{ab}\\
u_2\sqrt{b}-u_3\sqrt{ab} & u_0-u_1\sqrt{a}
\end{pmatrix},
\end{eqnarray*}
is a monomorphism of algebras (see \cite[Chapter X]{elstrodt}).
The group $\SL_{1}(\h(a,\,b, \mathfrak{o}_K))$ acts as orientation preserving isometries  on $\HQ^3$ via this embedding $u\mapsto \gamma_u$. The kernel of this action is  ${\rm{I}}(a,b,K)=\SL_{1}(\h(a,\,b, \mathfrak{o}_K))\cap \U (\mathfrak{o}_K)$, a finite group. We denote by $\PSL_{1}(\h(a,\,b, \mathfrak{o}_K)) =\SL_{1}(\h(a,\,b, \mathfrak{o}_K))/{\rm{I}}(a,b,K)$.  Hence finding a set of generators for (a subgroup of finite index in) $\SL_{1}(\h(a,\,b, \mathfrak{o}_K))$ amounts to finding  a set of generators for $\PSL_{1}(\h(a,\,b, \mathfrak{o}_K))$.

Let $\Gamma$ be a discrete subgroup of $PSL_2(\C)$. The Poincar\'e method  can be used to give a presentation of $\Gamma$ (see for example \cite[Chapter 6]{ratcliffe}). In particular the following corollary of Poincar\'e's theorem gives generators for $\Gamma$. Recall that a convex polyhedron in a metric space is a nonempty, closed, convex subset of the space such that the collection of its sides is locally finite, where a side is defined as a maximal convex subset of the border of the polyhedron.

\begin{theorem}\label{Poincare}\cite[Theorem 6.8.3]{ratcliffe}
Let $\F$ be a convex, fundamental polyhedron for a discrete group $\Gamma$. Moreover suppose that for every side $S$ of $\F$, there exists an element $\gamma \in \Gamma$, such that $S=\F \cap \gamma(\F)$, i.e. $\F$ is exact. Then $\Gamma$ is generated by the set
\begin{equation*}
\lbrace \gamma \in \Gamma \mid \F \cap \gamma(\F) \textrm{ is a side of } \F \rbrace.
\end{equation*}
The elements $\gamma$ of the generating set are called side-pairing transformations.
\end{theorem}

Recall that a set $\F$ is a fundamental domain for a discrete group $\Gamma$ acting on a metric space $X$ if $\F$ is closed and connected, the members of $\lbrace \gamma(\F)^{\circ} \mid \gamma \in \Gamma \rbrace$ are mutually disjoint and  $X=\bigcup_{\gamma \in \Gamma} \gamma(\F)$. A polyhedron that is a fundamental domain is called a fundamental polyhedron.

Recall that the hyperbolic distance $\rho$ in $\HQ^3$ is determined by
\begin{equation}
\cosh \rho(P,P')= \delta(P,P') = 1+\frac{d(P,P')^2}{2rr'},
\end{equation}
where $d$ is the Euclidean distance and $P=z+rj$ and $P'=z'+r'j$
are two elements of $\ \HQ^3$. We now recall the definition of Dirichlet fundamental polyhedron for a discrete subgroup $\Gamma$ of $\PSL_2(\C)$. Let $\gamma \in \Gamma$ and $P \in \HQ^3$ a point which is not fixed by $\gamma$. Then let
\begin{equation}\label{defbis}
D_{\gamma}(P)=\{u\in \HQ^3 \mid \rho(P,u)\leq \rho (u,\gamma (P))\},
\end{equation}
the half space containing $P$.
The bisector of $P$ and $\gamma(P)$ is the border of $D_{\gamma}(P)$, that is the set $\{u\in \HQ^3 \mid \rho(P,u)= \rho (u,\gamma (P))\}$. If $P \in \HQ^3$ is a point which is not fixed by any non-trivial element of $\Gamma$, then
\begin{equation*}
\F = \cap _{1 \neq \gamma \in \Gamma} D_{\gamma}(P)
\end{equation*}
is known as a Dirichlet fundamental polyhedron of $\Gamma$ with centre $P$. It is well known that the Dirichlet polyhedron is convex, exact and locally finite. If $\Gamma_P$, the stabilizer of $P$ in $\Gamma$, is not trivial and if $\F_P$ is a fundamental domain for the group $\Gamma_p$, then
\begin{equation}\label{fundwithstab}
\F = \F_P \cap ( \cap_{\gamma \in \Gamma \setminus \Gamma_P} D_{\gamma}(P))
\end{equation}
is a fundamental domain of $\Gamma$. A proof of this can be found for example in \cite[Proof of Proposition 3.2]{jesetall}.  So theoretically one has a method to compute generators for $\Gamma$. However for concrete classes of groups,
such as $\U(\h(a,b,\mathfrak{o}_k)$ and $\SL_2(\mathfrak{o}_k)$,
one would like to obtain an algorithm that determines in  finitely many steps  the above intersection and the side-pairing transformations.


\section{Towards an algorithm for computing a fundamental domain}
\label{seciso}

The main purpose of this section is to give a finite algorithm to compute a fundamental domain in $\HQ^3$ (respectively $\HQ^2$) and generators for a given discrete subgroup $\Gamma$ of $\PSL_2(\C)$ (respectively $\PSL_2(\R)$) that is of finite covolume. To do so we need explicit formulas for the bisectors defining the Dirichlet polyhedron associated to a discrete group in two and three dimensional hyperbolic space.  These formulas were obtained in \cite{jeskiefjuretall} but, for the readers convenience, we reproduce the crucial results needed to  obtain them. We then give some lemmas that simplify the algorithm. Finally we give an explicit criterion that determines the finite number of steps the algorithm has to go through. Calculations are done in dimension $3$. Standard facts about the theory of hyperbolic geometry will be used freely  (see for example \cite{beardon, bridson, elstrodt, gromov, ratcliffe}). In \cite{jesetall}
an algorithm was obtained
in the case $\Gamma$ is cocompact. Because of the explicit formulas obtained in \cite{jeskiefjuretall}, our algorithm is a refinement. Furthermore, our algorithm also applies to
the non-cocompact case.

Let $0\in \BQ^3$ be the origin and $\gamma =
\left(
\begin{array}{ll}
a & b \\ c & d
\end{array}
\right)\in \SL_2(\C )$ and  $\Psi (\gamma ) =
\left(
\begin{array}{ll}
A & C^{\prime} \\ C & A^{\prime}
\end{array}
\right) \in \SB_2(\h)$ (see Proposition \ref{propels}).
Recall that the isometric sphere associated to the transformation $\gamma$, respectively $\Psi(\gamma)$, is the unique sphere on
which $\gamma$, respectively $\Psi(\gamma)$ acts as a Euclidean isometry. It is well known (see for instance
\cite{beardon, ratcliffe}), that if $\gamma = \begin{pmatrix} a & b \\ c & d \end{pmatrix}$ with $c \neq 0$, then the isometric sphere of
$\gamma$ has centre $-\frac{d}{c}$ and radius $\frac{1}{\vert c \vert}$. An independent proof is given in \cite{jeskiefjuretall}, where it is also proved that if $\Psi(\gamma)=\begin{pmatrix} A &
C' \\ C & A' \end{pmatrix}$, then the centre and the radius of the isometric sphere of $\psi(\gamma)$ are respectively $-
C^{-1}A'$ and $\frac{1}{\vert C \vert}$. In the ball model we denote this isometric sphere by $\Sigma_{\Psi(\gamma)}$, its centre by $P_{\Psi (\gamma )}$ and its radius by
$R_{\Psi (\gamma )}$. Note that in fact the isometric sphere of the ball model is strictly speaking only the part of the Euclidean sphere $\Sigma_{\Psi(\gamma)}$ which intersects the ball model, i.e. $\BQ^3 \cap \Sigma_{\Psi(\gamma)}$. However throughout the paper we make some abuse of notation and denote the Euclidean sphere with centre $-C^{-1}A'$ and radius $\frac{1}{\vert C \vert}$ as well as the isometric sphere by $\Sigma_{\Psi(\gamma)}$. The following theorem shows that in the ball model the concepts of isometric
sphere of $\Psi(\gamma)$ and bisector of $0$ and $\Psi(\gamma^{-1})(0)$ are the same. Therefore let $\SU_2(\C)$ be the group of unitary matrices. Note that $\gamma \notin \SU_{2}(\C)$ if and only if $\gamma (j)=j$ (see \cite[Theorem 4.2.2]{beardon}) or, equivalently, $\Psi (\gamma )(0)=0$. In the latter case the bisector of $0$ and $\Psi(\gamma^{-1})(0)$ does not exist as both are the same points. Also $\Psi (\gamma )(0)=0$ if and only if $C=0$ and hence the isometric sphere does not exist neither.
Therefore in the following theorem we exclude the case $\gamma \in \SU_2(\C)$. Also, recall that another equivalent condition for $\gamma \notin \SU_{2}(\C )$ is $\| \gamma \|^2 \neq 2$ and vice-versa.

\begin{theorem}\cite{jeskiefjuretall}\label{isosphereB3}
Let $\gamma \in \SL_2(\C )$ with $ \gamma\notin \SU_2(\C)$.
Then, in the ball model, the isometric sphere associated to $\Psi(\gamma)$ equals the bisector of the geodesic segment linking $0$ and $\Psi(\gamma^{-1})(0)$. Moreover the centre $P_{\Psi(\gamma )}$ of $\Sigma_{\Psi(\gamma)}$ and $\Psi (\gamma^{-1})(0)$ are inverse points with respect to $S^2=\partial {\mathbb{B}}$.
\end{theorem}

\begin{proof}
Let $\Psi(\gamma)=\begin{pmatrix} A & C' \\ C & A' \end{pmatrix}$. So $P_{\Psi (\gamma )}= -C^{-1}A'$.
It is easily seen that $0$, $P_{\Psi (\gamma )}$ and $\Psi (\gamma^{-1})(0)$
are collinear. Indeed , one may easily show that $P_{\Psi (\gamma )} \cdot (\Psi(\gamma^{-1})(0))^{-1} = \vert A \vert ^2 \vert C \vert ^{-2} \in {\mathbb{R}}$ (as $C \neq 0$ by the assumption that $\gamma \not\in \SU_2(\C)$).
Moreover $\Vert P_{\Psi (\gamma )} \Vert \cdot \Vert (\Psi(\gamma^{-1})(0))^{-1} \Vert
= \vert -C^{-1}A^{\prime }\vert \cdot \vert -\overline{C}A^{*-1} \vert = 1$, where $\Vert V \Vert$ denotes the Euclidean norm of a vector $V$, and thus
$P_{\Psi(\gamma )}$ and $\Psi (\gamma^{-1})(0)$ are inverse points with
respect to $S^2=\partial {\mathbb{B}}$.

Now let $r$ be the ray through $P_{\Psi (\gamma )}$ and put $M= r\cap
\Sigma_{\Psi (\gamma )}$. Clearly $\|M\|=\frac{|A|-1}{|C|}$. Since the
hyperbolic metric, $\rho$, in ${\mathbb{B}^3}$ satisfies $\rho (0, u)=\ln(
\frac{1+\|u\|}{1-\|u\|})$ (see for instance \cite[Chapter 7.2]{beardon}), we have
\begin{eqnarray*}
\rho(0,M) =\ln(\frac{1+\frac{\vert A \vert -1}{\vert C \vert}}{1-\frac{\vert
A \vert -1}{\vert C \vert}}) = \ln(\frac{\vert C \vert + \vert A \vert -1}{
\vert C \vert - \vert A \vert +1}),
\end{eqnarray*}
and
\begin{eqnarray*}
\rho(0,\Psi(\gamma^{-1})(0)) =\ln(
\frac{\vert A \vert + \vert C \vert}{\vert A \vert - \vert C \vert}).
\end{eqnarray*}
Moreover, using the fact that $\vert A \vert^2 - \vert C \vert^2=1$, one easily calculates that $2\cdot  \ln(\frac{\vert C \vert + \vert A \vert -1}{
\vert C \vert - \vert A \vert +1})=\ln(\frac{\vert A \vert + \vert C \vert}{\vert A \vert - \vert C \vert})$, and hence $2\rho(0,M)=\rho(0,\Psi(\gamma^{-1})(0))$. Hence $M$ is the midpoint of the ray from $0$ to $\Psi (\gamma^{-1}) (0)$.
The ray $r$ being orthogonal to $\Sigma_{\Psi(\gamma)} $, it follows
that $\Sigma_{\Psi(\gamma)}$ is the bisector of the geodesic segment linking $0$ and $
\Psi(\gamma^{-1})(0)$.
\end{proof}

In the upper half space model, the result from the above theorem is not necessarily true. The bisector of the geodesic segment linking $0$ and $\Psi(\gamma^{-1})(0)$ being a pure hyperbolic notion, its projection $\eta_0^{-1}({\Sigma_{\Psi (\gamma)}})$ to the upper half space model $\HQ^3$ is still a bisector. In fact it is the bisector of the geodesic segment linking $\eta_0^{-1}(0)=j$ and $\eta_0^{-1}(\Psi(\gamma^{-1})(0))=\gamma^{-1}(j)$. Note that the latter equality comes from Proposition \ref{propels}.(v). However the isometric sphere associated to a transformation $\gamma$ is a purely Euclidean concept and hence the projection of the isometric sphere of $\Psi(\gamma)$ by $\eta_0^{-1}$ is no longer an isometric sphere. So, if we denote the isometric sphere in $\HQ^3$ associated to $\gamma$ by $\Iso_{\gamma}$, then in general we do not have that $\Iso_{\gamma}=\eta_0^{-1}({\Sigma_{\Psi (\gamma
)}})$. We put
$\Sigma_{\gamma}= \eta_0^{-1}({\Sigma_{\Psi
(\gamma )}})$. In fact, $\Sigma_{\gamma} \cap \HQ^3$ is nothing else then $D_{\gamma^{-1}}(j)$ (see the definition in (\ref{defbis})). Note that $\eta_0$ being an isometry yields that $\Sigma_{\gamma}$ is either a Euclidean sphere with centre in $\partial \HQ^3$ or a vertical plane. In case it is an Euclidean sphere,  we  denote its center by
$P_{\gamma}$  and its radius by $R_{\gamma}$.

\begin{proposition}\cite{jeskiefjuretall}
\label{isogammaupmodel}

Let $\gamma =
\left(
\begin{array}{ll}
a & b \\ c & d
\end{array}
\right) \in \SL_2(\C)$ and $\gamma \not \in \SU_2(\C)$.
\begin{enumerate}
\item
$\Sigma_{\gamma} $ is an Euclidean sphere if and
only if $|a|^2+|c|^2 \neq 1$. In this case, its
center  and its radius are respectively given by
$P_{\gamma}=\frac{-(\overline{a}b+\overline{c}d)}{|a|^2+|c|^2-1}$ and
$R^2_{\gamma}=\frac{1+\|P_{\gamma}\|^2}{|a|^2+|c|^2}$.
\item
$\Sigma_{\gamma}$ is a plane if and only if
$|a|^2+|c|^2 = 1$. In this case
$Re(\overline{v}z)+\frac{|v|^2}{2}=0$, $z\in \C$ is
a  defining equation of $\Sigma_{\gamma}$,  where
$v=\overline{a}b+\overline{c}d$.
\end{enumerate}
\end{proposition}

\begin{proof}  Suppose first that $\Sigma_{\gamma}$ is a Euclidean sphere. If we denote the inverse point of a point $P$ with respect to $S^2$ by $P^{*}$, then by Theorem \ref{isosphereB3}, $P_{\Psi (\gamma )}^* =
\Psi (\gamma )^{-1}(0)$. Consider the two spheres $\Sigma_{\Psi(\gamma)}$ and $S^2$. As
\begin{eqnarray*}
1 + \frac{1}{\vert C \vert^2} = \frac{1+\vert C \vert ^2}{\vert C \vert ^2}
= \frac{\vert A \vert ^2}{\vert C \vert ^2}= \vert P_{\Psi(\gamma )} \vert ^2,
\end{eqnarray*}
by Pythagoras' Theorem, the two spheres are orthogonal. Because of this orthogonality, $P_{\Psi (\gamma )}^* =
\Psi (\gamma )^{-1}(0)$ implies that $0$ and $\Psi(\gamma^{-1})(0)$ are inverse points with respect to $\Sigma_{\Psi(\gamma)}$.  By Proposition \ref{propels},  $\eta_0$ is a
 M\"{o}bius transformation. Hence it follows that
$j=\eta_0^{-1}(0)$ and
$\gamma^{-1}(j)=\eta_0^{-1}(\Psi (\gamma )^{-1}(0))$ are inverse points with respect to $\Sigma_{\gamma}=\eta_0^{-1}(\Sigma_{\Psi(\gamma)})$ (for details see \cite[Theorem 3.2.5]{beardon}).
So if, $\ \Sigma_{\gamma }$ is not a vertical
plane, then $j$, $\gamma^{-1}(j)$ and
$P_{\gamma}$ are collinear points. The exact expression of  $\gamma^{-1}(j)$ is
$-\frac{(\overline{a}b+\overline{c}d)}{|a|^2+|c|^2}+
\frac{1}{|a|^2+|c|^2}j$ and hence for the three points to be collinear we must have that $|a|^2+|c|^2 \neq 1$.  In that case it follows that
$P_{\gamma}= l\cap \partial \HQ^3$, where $l$ is
the Euclidean line determined by $j$ and $\gamma^{-1}(j)$.
A simple calculation
gives the formula of $\ P_{\gamma}$. Since $j$ and $\gamma^{-1}(j)$ are inverse points with respect to $\Sigma_{\gamma}$, $R^2_{\gamma}=\|j-P_{\gamma}\|\cdot
\| \gamma^{-1}(j)-P_{\gamma}\|$. This gives the formula of $R^2_{\gamma}$ and proves the first item.

If $|a|^2+|c|^2=1$, the line $l$ determined by $j$ and $\gamma^{-1}(j)$ is parallel to the border of the upper half space $\partial\HQ^3$. So $\Sigma_{\gamma}$ cannot be a Euclidean sphere and hence is a vertical plane. Conversely, if $\Sigma_{\gamma}$ is a vertical plane, $j$ and $\gamma^{-1}(j)$ have to be at the same height and hence $|a|^2+|c|^2=1$. In this case,
$\gamma^{-1}(j)=
-(\overline{a}b+\overline{c}d)+j$ and hence
$v=j - \gamma^{-1}(j)= \overline{a}b+\overline{c}d$ is orthogonal to
$\Sigma_{\gamma}$. From this one obtains the
mentioned defining equation of $\ \Sigma_{\gamma}$, hence the second item.
\end{proof}

The next proposition gives some more information on the bisectors in the ball model.

\begin{proposition}\cite{jeskiefjuretall}
\label{isogammaballmodel}
Let $\gamma = \left(
\begin{array}{ll}
a & b \\ c & d
\end{array}
\right)\in \SL_2(\C )$ and  $\Psi (\gamma ) =
\left(
\begin{array}{ll}
A & C^{\prime} \\ C & A^{\prime}
\end{array}
\right)$. Suppose that $\gamma \notin \SU_{2}(\C)$.
Then the following properties hold.
\begin{enumerate}
\item $|A|^2=\frac{2+\|\gamma\|^2}{4}$ and $|C|^2=\frac{\|\gamma\|^2-2}{4}$
\item
$P_{\Psi (\gamma )}=
\frac{1}{-2+\|\gamma\|^2}\cdot [\
-2(\overline{a}b+\overline{c}d) +
[(|b|^2+|d|^2)-(|a|^2+|c|^2)]j\ ]$
\item
$\Psi (\gamma^{-1})(0)= P_{\Psi (\gamma )}^*=
\frac{1}{2+\|\gamma\|^2}\cdot [\
-2(\overline{a}b+\overline{c}d) +
[(|b|^2+|d|^2)-(|a|^2+|c|^2)]j\ ]$ (notation of
inverse point w.r.t. $S^2$).
\item
$\|P_{\Psi (\gamma )}\|^2 = \frac{2+\|\gamma\|^2}{-2+\|\gamma\|^2}$
\item
$ R_{\Psi (\gamma )}^2= \frac{4}{-2+\|\gamma\|^2}$
\end{enumerate}
\end{proposition}

\begin{proof}
The proof of the five items  is straightforward using the explicit
formulas for $\Psi (\gamma )$, $A$ and $C$ given by equation (\ref{psigamma}) and knowing that $P_{\Psi
(\gamma )}= C^{-1}A^{\prime}$.
\end{proof}

We will now give a lemma which will be crucial in the implementation of the algorithm. Let $\gamma \in \SL_2(\C )$ and $\gamma \notin \SU_2(\C)$ and let $r$ be  the
ray through the center of $\Sigma_{\Psi (\gamma
)}$. Denote by $M$ and $N$,  respectively,  the
intersection of $\ r$ with $\Sigma_{\Psi (\gamma
)}$ and $S^2$. Denote the Euclidean distance from $M$ to $N$ by
$\rho_{\gamma}$. Explicitly we have
that  $\rho_{\gamma} =1+R_{\Psi (\gamma
)}-\|P_{\Psi (\gamma )}\|$. Our next result
shows that  $\rho_{\gamma}$ is a strictly
decreasing function of $\|\gamma\|$. The proof of this result may be found in \cite[proof of Lemma 3.7]{jeskiefjuretall} ,
but for the convenience of the reader, we reproduce it here.
Note that the Euclidean  volume of the intersection of the interior of $\Sigma_{\Psi (\gamma )}$  with $\BQ^3$ is a function of $\rho_{\gamma}$.

\begin{lemma}\cite{jeskiefjuretall}
\label{volume}
Let $\Gamma$  be a discrete subgroup of $\PSL_2(\C)$ acting on $\BQ^3$.  Then $\rho_{\gamma}$ is a strictly decreasing function of $\ \|\gamma\|^2$  on $\Gamma \setminus \SU_2(\C)$.
\end{lemma}

\begin{proof}
Using Proposition \ref{isogammaballmodel}, one obtains that  $\rho_{\gamma}= 1-(\frac{\Vert \gamma \Vert^2+2}{\Vert \gamma \Vert^2-2})^{\frac{1}{2}}+2(\Vert \gamma \Vert^2-2)^{\frac{-1}{2}}$. It is well known that for any $\gamma \in \GL_2(\C)$, we have that $2 \cdot \vert \det (\gamma )\vert \leq \|\gamma\|^2$, and thus  $\Vert \gamma \Vert^2 \geq 2$ if $\gamma \in \SL_{2} ( \C )$,  with equality if and only if $\gamma \in \SU_2(\C)$.  Consider now the continuous function $f: \left]2,+\infty \right[ \longrightarrow \R$ given by $f(x)=1-(\frac{x^2+2}{x^2-2})^{\frac{1}{2}}+2(x^2-2)^{\frac{-1}{2}}$. Then
$f^{\prime}(x)=-2x(x^2-2)^{-3/2}(x^2+2)^{-1/2}[-2+\sqrt{x^2+2}]$, which shows that $f$ is a strictly decreasing function. From this the result follows.
\end{proof}

We now come to our algorithm to compute a Dirichlet fundamental polyhedron for  a discrete subgroup $\Gamma$  of $\mbox {Iso}^+(\HQ^3)$ which is  of finite covolume. Recall that a group $\Gamma$ is of finite covolume if the volume of its fundamental domain is finite. More information on hyperbolic volumes may be found in \cite[Chapter 3]{ratcliffe}. We describe this algorithm first under the assumption that the stabilizer $\Psi(\Gamma)_0$ of the point $0$ in $\BQ^3$ is trivial. In the next section we avoid this assumption and explain how to change the algorithm slightly depending on $\Gamma$ being the unit group of an order in a division algebra or a matrix algebra. Because of the concrete formulas obtained in Proposition~\ref{isogammaballmodel} we compute a fundamental domain in $\BQ^3$. Using the map $\eta_{0}$ one can then
convert this to  a fundamental domain in $\HQ^{3}$ which is more suitable for visualization.
For $\gamma \in \SL_2(\C)$, define $B(\gamma)$ to be the intersection of the Euclidean ball $\overline{B}_{R_{\Psi (\gamma )}}(P_{\Psi (\gamma )})$ with the closed unit ball $\overline{\BQ^3}$ (so this full closed ball determined by the sphere $\Sigma_{\Psi (\gamma )}$). Let $f:\Gamma \rightarrow \C$ be the map defined by $f(\gamma)= \Vert \gamma \Vert^2$. Then we order the elements of $Im(f)$ in  a strictly increasing sequence $r_i$ for $i \geq 1$. Note that this is possible because of the discreteness of $\Gamma$. For $n \geq 1$, we define the sets $\F_{n}$ recursively in the following way:
\begin{equation*}
\F_1=\bigcup_{\gamma \in \Gamma} \lbrace B(\gamma) \mid \textrm{ and } \Vert \gamma \Vert ^2 =r_1 \textrm{ and } \gamma \neq 1 \rbrace
\end{equation*}
and for every $n \geq 1$ define
\begin{equation*}
\F_n=\bigcup_{i=r_1}^{r_{n-1}} \F_i \cup \bigcup_{\gamma \in \Gamma} \lbrace B(\gamma) \mid \Vert \gamma \Vert ^2 =r_n \textrm{ and } B(\gamma) \not\subseteq \cup_{i=r_1}^{r_{n-1}} \F_i \rbrace.
\end{equation*}
Note that later in all the examples we consider, the sequence $r_i$ with $i \geq 1$ may be taken inside the set of natural numbers.
Using these definitions, the following proposition describes the \textit{Dirichlet Algorithm  of Finite Covolume} (DAFC). Note that we state the algorithm here under the condition that the stabilizer of $0$ in the ball model, $\Psi(\Gamma)_0$ is trivial. However in the beginning of the next section we explain how one may adapt this algorithm to cases with non trivial stabilizer.

\begin{proposition}[DAFC]\label{DAFC} Let $\Gamma \subseteq \PSL_2(\C)$ be a discrete group of finite covolume and with $\Psi(\Gamma)_{0}$ trivial. Then the following  algorithm computes in a finite number of steps in $\HQ^3$ the Dirichlet fundamental polyhedron with centre $j$ associated to a subgroup of finite index of $\Gamma$.
\begin{enumerate}[Step 1:]
\item Compute $\F_1$, $\F_2$, $\ldots$ in this given order.
\item Set $N$ the minimum such that $\partial \BQ^3 \subseteq  \F_N$.
\item $\F=\overline{ \BQ^3 \setminus \F_{N}}$ is a fundamental polyhedron in $\BQ^3$.
\item Use $\eta_0^{-1}$ to obtain in $\HQ^{3}$ the  fundamental polyhedron $\bigcap_{\gamma \in L} D_{\gamma^{-1}} (j)$, where $L=\{ \gamma \in \Gamma \mid B(\gamma ) \in \F_{N} \}$ (a finite set). The intersection is non-redundant.
\end{enumerate}
Moreover, $\langle \gamma \mid \gamma \in L \rangle$ is a subgroup of finite index in $\Gamma$.
\end{proposition}

\begin{proof}
We claim that for every $n \in \N$, the set $\lbrace B(\gamma) \mid \Vert \gamma \Vert ^2 =n \rbrace$ is finite.
Indeed if $\|\gamma_1\|=\|\gamma_2\|$ then, by Proposition \ref{isogammaballmodel},  $\| P_{\Psi(\gamma_1)}^*\|= \| P_{\Psi (\gamma_2)}^*\|$ and hence also $\Vert \Psi(\gamma_1^{-1})(0) \Vert= \Vert \Psi(\gamma_2^{-1})(0) \Vert$. Since $\Gamma$ is discrete and the ball $\overline{B}_{\Vert \Psi(\gamma_1^{-1})(0) \Vert}(0)$ in $\BQ^3$  is compact, the claim follows. So the sets $\F_i$ for $i \geq 1$ are also finite and thus they are computable in a finite number of steps. Since $\Gamma$ is of finite covolume, results of Greenberg, Garland and Raghunathan (see \cite[Theorem II.2.7]{elstrodt}) imply that the Dirichlet fundamental domain of $\Gamma$ has finitely many sides and thus the DAFC  stops in  a finite number of steps. Define the set $L=\lbrace \gamma \in \Gamma \mid B(\gamma) \in \F_N \rbrace$. Then the fundamental domain $\F$ given by the DAFC is the finite intersection $\bigcap_{\gamma \in L} D_{\gamma^{-1}}(j)$. Lemma \ref{volume} guarantees that this intersection is non-redundant. This finishes the proof of the algorithm.

Clearly, by definition, the Dirichlet fundamental domain with centre $j$ is given by the intersections of the sets  $D_{\gamma}(j)=\{u\in \HQ^3 \mid \rho(j,u)\leq \rho (u,\gamma (j))\}$ for finitely many $\gamma \in \Gamma$. It is well known (see \cite[Theorem 6.7.4]{ratcliffe}) that those finitely many $\gamma$ are exactly the side pairing transformations and hence generate the group $\Gamma$. So the set of $\gamma \in \Gamma$ such that $B(\gamma) \in \F_N$ is a generating set for a subgroup of finite index in $\Gamma$.
\end{proof}

Note that the fundamental domain computed by the DAFC is not necessarily the fundamental domain of the complete group $\Gamma$. In fact the algorithm stops as soon as it has found a fundamental domain of finite volume. By the Poincar\'e Method, this just guarantees that we are dealing with a fundamental domain of a subgroup of finite index in $\Gamma$. However, this is sufficient for the main purpose of this paper: obtaining  finitely many generators for a subgroup of finite index in $\U(\Z G)$. Nevertheless, the next proposition also gives a finite  algorithm  to compute a fundamental domain of the complete group $\Gamma$.

\begin{proposition}[Refined DAFC]\label{completeset}
Suppose $\F=\eta_0^{-1}(\overline{\BQ^3 \setminus \F_N})$ is the fundamental domain in $\HQ^3$ given by the DAFC. Then the following finite algorithm gives a fundamental domain of the complete group $\Gamma$.
\begin{enumerate}[Step 1:]
\item Compute the finite number of vertices $V_i$ of $\F$.
\item Compute $k=\textrm{cosh}^{-1}(\frac{r_N}{2})$, $r=\textrm{max}\lbrace \lbrace \frac{k}{2}   \rbrace \cup \lbrace \rho(j,V_i) \mid V_i \textrm{ vertex of } \F \rbrace \rbrace$ and
$\tilde{N} = 2 cosh (2r)$.
\item $\tilde{\Fe}=\eta_0^{-1}(\overline{\BQ^3 \setminus \F_{\tilde{N}}})$ is a fundamental domain for $\Gamma$.
\end{enumerate}
Moreover, $\Gamma =\langle \gamma \mid B(\gamma ) \subseteq \F_{\tilde{N}} \rangle$.
\end{proposition}

\begin{proof}
A simple computation shows that $k=\textrm{max}\lbrace \rho(\gamma(j),j) \mid \Vert \gamma \Vert^2 \leq r_N \rbrace$. As in the proof of Proposition 3.1 in \cite{jesetall}, one shows
that  $\tilde{\F}$ is a fundamental domain of the  group $\Gamma$.
\end{proof}

Note that for our application to the unit group of integral group rings, the groups we are working with will be discrete subgroups of unit groups of orders in quaternion division algebras or some discrete subgroups of $\SL_2(\C)$. The DAFC cannot be directly applied to such groups, but only to their projections in $\PSL_2(\C)$. More concretely, if one is interested in finding generators for a discrete subgroup $\Gamma$ of $\SL_2(\C)$, one may use the DAFC to get generators for the projection of $\Gamma$ in $\PSL_2(\C)$.
The generators of $\Gamma$ will then be the pre images in $\SL_{2}(\C )$ of the discovered units together with the (finite) kernel of the action of $\Gamma$ on $\HQ^{3}$.
If, on the other hand,  $\Gamma$ is the discrete group  $\SL_{1}(\h(a,\,b, \mathfrak{o}_K))$ determined by a (division) quaternion algebra $\h (a,b, K)$, then one has to proceed as explained in section \ref{secback}: first one has to embed $\SL_{1}(\h(a,\,b, \mathfrak{o}_K))$ in $\SL_2(\C)$ and then consider its image in $\PSL_2(\C)$. To get a set of generators, one has to add the kernel of the action (i.e. ${\rm{I}}(a,b,K)$) to the pre images  of the set of generators given by the DAFC.

As will be shown in several applications, often Dirichlet fundamental polyhedra of discrete groups contain symmetries. These can be used to shorten the DAFC and to list the generators of the discrete group in a more compact manner (see the next section for details). The following proposition of \cite{jeskiefjuretall} that describes some isomorphisms and some involutions of $\PSL_2(\C)$ will be useful to describe some symmetries for the applications under consideration. 

\begin{proposition}\cite{jeskiefjuretall}\label{symmetries}
Let $\gamma = \begin{pmatrix} a& b\\ c& d\end{pmatrix} \in \SL_2(\C) \setminus \SU_2(\C)$ with $|a|^{2}+|c|^{2}\neq 1$ (so $\Sigma_{\gamma}$ is a Euclidean sphere by
Proposition~\ref{isogammaupmodel}).
Denote by $\sigma$ the conjugation by the matrix $\begin{pmatrix}
\sqrt{i} & 0\\
0 & \sqrt{-i}
\end{pmatrix}
$, by $\delta$ the conjugation by the matrix $ \begin{pmatrix}
0 & -1\\
1 & 0
\end{pmatrix}$. Let $\tau (\gamma ) =\overline{\gamma}$ denote complex conjugation of the entries of $\gamma$ and define $\phi=\sigma^2\circ \delta\circ\tau$. Then in $\HQ^3$
\begin{enumerate}
\item $P_{\phi (\gamma)}$ is the reflection of $\ P_{\gamma}$ in $S^2$,
\item $\tau$ induces a reflection in the plane spanned by $1$ and $j$, i.e. $P_{\tau (\gamma )}=\overline{P_{\gamma}}$ and $R_{\tau (\gamma )}=R_{\gamma}$ and
\item $\sigma^2$ induces a reflection in the origin, i.e. $P_{\sigma^2 (\gamma )}=-P_{\gamma}$ and $R_{\sigma^2 (\gamma )}=R_{\gamma}$.
\item $\sigma$ restricted to $\partial \HQ^3=\lbrace z  \in \C \rbrace$ induces a rotation of ninety degrees around the point of origin, i.e. $P_{\sigma(\gamma )}=iP_{\gamma}$ and $R_{\sigma(\gamma )}=R_{\gamma}$.
\end{enumerate}
\end{proposition}

\begin{proof}
To prove the first item, first note that $P_{\Psi (\phi (\gamma ))}=\pi (P_{\Psi (\gamma )})$, where $\pi$ is the reflection in the plane $\lbrace (x,y,z) \in \R^3 \mid z=0 \rbrace$. Since $\eta_0:\HQ^3\rightarrow \BQ^3$ is equivariant, it follows that $P_{\phi (\gamma)}$ is the reflection of $\ P_{\gamma}$ in $S^2$, but the radius is not necessarily maintained.
The next three items follow from mere calculations.
\end{proof}

\begin{remark}
In $\HQ^2$, $\phi$ is a reflection in $S^1$, $\tau$ is a reflection in the imaginary axis  and $\sigma^2$ has the same action on $P_{\gamma}$ as
 $\tau$ does, for every $\gamma \in SL_{2}(\C)$.
\end{remark}

Finally the following result of \cite{jeskiefjuretall} will be useful when implementing the algorithm for a given group. Note that it gives an easily verifiable criteria for cocompactness. Its proof follows easily from Proposition \ref{isogammaballmodel}.

\begin{lemma}\cite{jeskiefjuretall}
\label{controlfundom}
Let $\gamma = \left(
\begin{array}{ll}
a & b \\ c & d
\end{array}
\right)\in \SL_2(\C )$ and $\gamma \notin \SU_2(\C)$. Then
\begin{enumerate}
\item $0\notin \Sigma_{\Psi (\gamma )}$.
\item $j\in \Sigma_{\Psi (\gamma )}$ if and only if $\ |a|^2+|c|^2=1$.
\item $j\in { {\rm Interior}} (\Sigma_{\Psi (\gamma )})$ if and only if $\ |a|^2+|c|^2<1$.
\item $-j\in \Sigma_{\Psi (\gamma )}$ if and only if $\ |b|^2+|d|^2=1$.
\item $-j\in {{\rm Interior}}(\Sigma_{\Psi (\gamma )})$ if and only if $\ |b|^2+|d|^2<1$.
\end{enumerate}
\end{lemma}

Note that all the above lemmas and propositions may be established in the two dimensional model. Let $\gamma\in \SL_2(\R )$. As a M\"{o}bius
transformation, $\gamma$ acts on $\HQ^2$ and  $\eta_0$ given by the matrix $\begin{pmatrix}
1 & -i \\ -i & 1
\end{pmatrix}$ gives an
isometry between the two models. Proceeding as in the
$3$-dimensional model, we  obtain explicit
formulas for the bisectors. The role of $\ j$ is played by $i$ and, since we are in the commutative setting, calculations are easier.


\section{Applications}
\label{secappl}

In this section we apply the DAFC to several examples. We divide this section in two subsections. First we treat examples coming from division algebras and hence ameliorate \cite{jesetall}. In fact in \cite{jesetall}, the authors were only able to treat small examples, because of absence of concrete formulas. Then in the second subsection we apply our method to matrix algebras and hence show that our method does not require the condition of cocompactness, which was required in \cite{jesetall}.

\subsection{Division Algebras}\label{divalg}

All examples given are  discrete and cocompact subgroups of ${\rm{Iso}}^+(\BQ^n), n=2,3$. In fact we could just apply the DAFC to the examples. However as these groups are cocompact, we get a sort of "starting point", which will make the implementation much easier. To simplify notations, we will from now on omit the notation $\Gamma$ and denote by $\Gamma$ as well the group $\Gamma$ as its embedding in $\SL_2(\C)$. One encounters two situations.
\vp

\noindent {\bf{Case I:}} {\bf{$\Psi(\Gamma)_0$ is trivial}}. In this case we may just apply the DAFC such as it is stated in Proposition \ref{DAFC} (respectively \ref{completeset}). However Lemma \ref{controlfundom} and the fact that the group is cocompact gives us a way of finding a "special" bisector with which we may start. In fact as the group is cocompact, there has to exist $\gamma_0\in \Gamma$ whose bisector $\Sigma_{\Psi(\gamma_0)}$ separates $j$ and the origin $0$. Because of Lemma~\ref{controlfundom} $\gamma_0= \left(
\begin{array}{ll}
a & b \\ c & d
\end{array}
\right)  $ with $|a|^2+|c|^2<1$. Thus we may look for such a $\gamma_0$ with smallest norm and set $V=\mbox{Exterior} (\Sigma_{\Psi(\gamma_0)})\cap \partial \BQ^3$. Then the DAFC stops when $V\subseteq \F_N$, for some $N \in \N$. As stated before this is absolutely not necessary for the DAFC to run. However if one considers the situation in the upper half space model, instead of the ball model, one notices that, because of cocompactness, there has to exist one bisector separating the points $j$ and $\infty$. In fact the upper half space model being not as symmetric as the ball model, this special bisector represents some kind of "upper dome" (see Figures \ref{d15}, \ref{d23} or \ref{25qi} for example) which guarantees the fundamental domain to be compact.
\vp

\noindent {\bf{Case II:}}{\bf{$\Psi(\Gamma)_0$ is non-trivial}}. In this case we first determine a fundamental domain ${\mathcal{F}}_0$  of $\Psi(\Gamma)_0$,  a polyhedron with the origin on its boundary. In all the applications we consider, the stabilizer $\Psi(\Gamma)_0$ is a small group and hence $\F_0$ is relatively easy to determine. As the fundamental domain is given by the intersection of some construction, based on a Dirichlet fundamental domain, with $\F_0$, we can modify the DAFC in such a way that it does not stop when the whole border of $\BQ^n$ is covered by the different $F_i$ for $i \geq 2$, but earlier. In fact let $V=\partial ( \mathcal{F}_0)\cap \partial \BQ^n$. Moreover the definition of the $\F_i$ has to be adapted in such a way that one only considers $B(\gamma)$ for $\gamma \in \Gamma \setminus \Psi^{-1}(\Psi(\Gamma)_0)$, and not in the whole group $\Gamma$. Then we will find a fundamental domain for a subgroup of finite index of $\Gamma$, as well as a generating set for this subgroup, by letting the DAFC stop when $V\subseteq \F_{N}$, for $N \in \N$. As  output we get  $\langle G_0,\gamma_1,\cdots , \gamma_n\rangle$ which has finite index in  $\Gamma$ and where $G_0$ denotes a generating set for $\Psi(\Gamma)_0$ and $\gamma_i$ for $1 \leq i \leq n$ are such that $B(\gamma_i) \in \F_N$.
\vp

\noindent {\bf{Generators for $ \Gamma$:}} Once obtained a fundamental domain $\mathcal{F}'$, say, of a subgroup of finite index in $\Gamma$, one finds a fundamental domain for $\Gamma$ by applying the refined DAFC of Proposition \ref{completeset} and hence also a set of generators.
\vp

We now have all the tools to implement  the DAFC in a way that the reader can follow and easily reproduce.
\vp

Our first example is a generalized quaternion division algebra $\h(a,b, K)$ with $K=\Q (\sqrt{-d})$, with $d$ a positive square free integer. In this case, the unit group of any order of $\h(a,b,K)$  is a cocompact Kleinian group (see \cite[Theorem X.1.2]{elstrodt}). In particular, we revisit the work of \cite{jesetall}, we consider Example 8 of chapter X of \cite{elstrodt} (page  474), and one example in dimension $2$.

To implement the DAFC, we find an additive basis for the ring of integers of $\ K$. This, together with the fact that we only consider reduced norm one elements in the unit group,  leads to a system of Diophantine equations whose  solution set is $\Gamma$. We get a sieve parameterizing the system by the matrix norm of the elements. For small values of $d$ the algorithm easily can be done by hand, and for larger $d$, we made use of the  software package Mathematica,  to get a Dirichlet fundamental domain and hence a set of generators (up to finite index) (cf. the homepage of the third author).

We first revisit \cite{jesetall}. In this case we get a particular nice system of Diophantine equations one of which is to write a number as the sum of four squares.
\vp

Consider $\h(K)=\h(-1,-1,K)$, with $K
=\Q (\sqrt{-d})$, with $d$ a positive square free
integer. We will consider the cases $d=15$ and $d=23$. Suppose that $d\equiv 3\  \mbox{mod}\
4$.  When $d\equiv 7\  \mbox{mod}\  8$ then
$\h(K)$ is a division algebra and thus $\Gamma
=SL_{1}(\h(-1,\,-1,\mathfrak{o}_K))=\{x
\in \h(-1,\,-1,\mathfrak o_K):N(x)=1\},$ acts discretely and cocompactly
on $\HQ^3$. In all other cases $\h(K)$ is a
matrix algebra and one may show  that in
these cases $\Gamma$ is never cocompact (see also
 Chapter VII of \cite{elstrodt}).

We have that
$\Psi(\Gamma)_0 =\Psi (\Gamma_{j})$ and it is easily seen that $\Gamma_{j}=\langle i,j\rangle \cong Q_8$, the
quaternion group of order $8$ and  that $\eta_0^{-1}(\Fe_0)$  can be taken to be that part of the unit ball centered at the origin whose projection on $\partial \HQ^3$ is the upper half of the unit circle, $\{z\in \C \ | \ |z|\leq1, Im(z)\geq 0\}$.

We have that $\mathfrak
o_K=\mbox{span}_{\Z}[1,w]$, where
$w=\frac{1+\sqrt{-d}}{2}$. Write
$u=u_0+u_1i+u_2j+u_3k\in \h (K )$ and
$u_t=x_t+y_tw$. Define $x=(x_0,x_1,x_2,x_3)$ and
$y=(y_0,y_1,y_2,y_3)$. So $u$ is determined by the vector $(x,y)$. Furthermore, in this example we use the embedding $u\mapsto  \gamma_u = \left(
\begin{array}{ll}
u_0+u_1i & u_2+u_3i \\ -u_2+u_3i &  u_0-u_1i
\end{array}\right)$ as used in \cite{jesetall}.

The next lemma gives formulas to compute $\Vert \gamma_u \Vert$ as well as the bisector associated to an element $\gamma_u$. These concepts are necessary to implement the DAFC.

\begin{lemma}
\label{unitsdivision} Let
$u=u_0+u_1i+u_2j+u_3k\in \h (K )$ and $N(u)= \pm 1$. Then
 \begin{eqnarray}
\begin{cases}
\|x\|^2-(\frac{d+1}{4})\|y\|^2=\pm 1\\
2\langle x|y\rangle +\|y\|^2=0.
\end{cases}
\end{eqnarray}
Moreover, the following hold.
\begin{enumerate}[(i)]
  \item $\| \gamma_u \|^2= 2N(u)+d\|y\|^2\in 2\Z$ \label{normgamma}
  \item $\|P_{\Psi (\gamma )}\|^2=\frac{2+2N(u)+d\|y\|^2}{-2+2N(u)+d\|y\|^2}$
  \item  $R^2_{\Psi (\gamma )}= \frac{4}{-2+2N(u)+d\|y\|^2}$.
  \end{enumerate}

  \end{lemma}

\begin{proof} The first part of the statement follows by  using the integral basis and the condition on the reduced norm. The second part follows from Proposition~\ref{isogammaballmodel}. \end{proof}

By Lemma \ref{unitsdivision} item (\ref{normgamma}), the norm of $\gamma_u$ only depends on $\Vert y \Vert^2$ and hence we may order the different sets $\F_i$ appearing in the DAFC by $\Vert y \Vert^2$ instead of by the norm of $\gamma_u$. Also because of the first two defining equations of Lemma \ref{unitsdivision}, we may define the sequence $r_n$ used in the definition of the sets $\F_n$ appearing in Proposition \ref{DAFC} as follows: for $n \geq 1$, $r_n=2n$. So the definition of the new sets $\F'_n$ will be as follows.

\begin{eqnarray}\label{deffi}
\F'_1 & = & \bigcup_{\gamma \in \Gamma} \lbrace B(\gamma) \mid \textrm{ and } \Vert y \Vert ^2 =2 \textrm{ and } \gamma \neq 1 \rbrace \\
\F'_n & = & \bigcup_{i=r_1}^{r_{n-1}} \F'_i \cup \bigcup_{\gamma \in \Gamma} \lbrace B(\gamma) \mid \Vert y \Vert ^2 =2n \textrm{ and } B(\gamma) \not\subseteq \cup_{i=r_1}^{r_{n-1}} \F'_i \rbrace.
\end{eqnarray}

Let $J(x)=(x_1,-x_0,-x_3,x_2)$ and
$S(x)=(-x_3,x_2,-x_1,x_0)$ and $J(y)$ is analogously defined. Then $J$ and $S$ are
skew orthogonal linear maps and $\langle S,J\rangle$ is isomorphic to $Q_8$.

\begin{lemma}
\label{ballcase} Let $u\in \h (K ), N(u)=\pm 1$
and $\gamma = \gamma_u = \left(
\begin{array}{ll}
u_0+u_1i & u_2+u_3i \\ -u_2+u_3i &  u_0-u_1i
\end{array}\right)= \left(
\begin{array}{ll}
a & b \\ c &  d
\end{array}\right)$. Then
\begin{enumerate}
\item
$|a|^2+|c|^2=N(u)+\frac{d}{2}\|y\|^2+\langle
J(x)|y\rangle\cdot \sqrt{d}$ and
$|b|^2+|d|^2=N(u)+\frac{d}{2}\|y\|^2-\langle
J(x)|y\rangle\cdot \sqrt{d}$
\item
$\overline{a}b+\overline{c}d=[\langle
-S(x)|y\rangle +i\langle SJ(x)|y\rangle]\cdot
\sqrt{d}$
\item
$|\overline{a}b+\overline{c}d|^2=
(N(u)+\frac{d}{2}\| y\|^2)^2-d\langle
J(x)|y\rangle^2 -N(u)$

\item If $N(u)=1$ then $R_{\Psi (\gamma)} =\frac{2\sqrt{d}}{d\| y\|}$

\item If $N(u)=1$ then
$P_{\Psi(\gamma)}=\frac{2\sqrt{d}}{d\|
y\|^2}\cdot [\langle S(x)|y\rangle -\langle
S(x)|J(y)\rangle i-\langle J(x)|y\rangle j\ ]$

\end{enumerate}
\end{lemma}

The formulas show that the centers of the bisectors in the ball model, up to a scalar in $\Z [\sqrt{d}]$, belong to $\Z^3$ and in the upper half space model they belong to  $\Z [i]$. The group of symmetries of a fundamental domain of $\Gamma$  contains $\langle \sigma, \tau\rangle$, where $\sigma$ and $\tau$ are as defined in Proposition \ref{symmetries}, and $\langle  \Gamma, \sigma , \tau\rangle$ is a discrete group. Note that Proposition \ref{symmetries} also indicates another symmetry which is inversion in $S^2$ and which is denoted by $\phi$. This symmetry also acts on the tessellation of $\HQ^3$ induced by $\Gamma$. However $S^2\cap \HQ^3$ is part of the boundary of the fundamental domain of $\Gamma_j$ and hence also part of the boundary of $\F$. Hence in this case this symmetry is lost. In the next example we will work with a group $\Gamma$ having a trivial stabilizer $\Gamma_j$ and hence the symmetry $\phi$ will show up.

With all this information we are ready to implement the DAFC. We do this for the cases $d=15$ and $d=23$. We analyse the case $d=15$ in details. As stated above, $\Psi ( \langle i, j \rangle )=\Psi (\Gamma)_{0}$  and $\eta_0^{-1}(\F_0)$ is as described above. We define $V=\partial \F_0 \cap\partial \BQ^3$ and hence the DAFC looks for $N$ minimal such that $V \subseteq \F_N$. Moreover by using the symmetries we can make the DAFC even faster. In the upper half space model this may be seen in the following way: $V \subseteq \F_N$ means in the upper half space model that the "base" of $\eta_0^{-1}(\F_0)$ is covered by Euclidean spheres. Mathematically this means that for every point $P \in \lbrace z+rj \in \HQ^3 \mid r=0,\ \vert z \vert ^2 \leq 1 \textrm{ and } Im(z)\geq 0 \rbrace$, there exists $\gamma \in \Gamma$ such that $P \in \eta_0^{-1}(B(\gamma))$. We compute the sets $F'_n$ for $n \geq 1$. However because of the symmetries we do not need to "cover" the whole "base" of $\eta_0^{-1}(\F_0)$, but only one quarter (as shown in figure \ref{d152dim}). In fact, by Proposition \ref{symmetries}, if we denote the set of spheres shown in Figure \ref{d152dim} by $X_{15}$, then $X_{15} \cup \sigma(X_{15}) \cup \tau\sigma^2(X_{15} \cup \sigma(X_{15}))$ "covers" the whole $\eta_0^{-1}(\F_0) \cap \partial \HQ^3$. Hence as a supplementary condition in the definition of $\F'_n$ we set that $B(\gamma) \cap (\eta_0^{-1}(\F_0) \cap \partial \HQ^3) \neq \emptyset$.  At $N=4$, the algorithm stops. Moreover including all the conditions we obtain $5$ different $B(\gamma)$ in $\F'_1$, no $B(\gamma)$ in $\F'_2$ nor in $\F'_3$ and finally $4$ more $B(\gamma)$ in $\F'_4$. These give us the $9$ different bisectors shown in Figure \ref{d152dim}. Thus the set $S_{15}=\Psi^{-1}(\Psi(\Gamma)_0) \cup \{g(\gamma) \mid g\in \langle \sigma , \tau
\rangle ,\ \eta_0^{-1}(B(\gamma)) \in X_{15}\}$ gives a generating set for a subgroup of finite index in $\Gamma$.
We may also apply the refined DAFC. Therefore we first compute $\max \lbrace \rho(j,V_i) \mid V_i \textrm{
vertex of } \F \rbrace$ which gives us $\sim 3.33$. We also compute $k$. By Proposition \ref{completeset},
the definition of $k$ is based on the maximal value $\Vert \gamma \Vert^2$ takes. In this case the maximal value is
$N=4$, which gives a maximal value $r_N=8$, which gives the maximal value for $\Vert y \Vert^2$. By Lemma
\ref{unitsdivision} we get a maximal value $\Vert \gamma \Vert^2=2+15\cdot 8=122$. This gives $k=cosh^{-1}(\frac{122}
{2})=\sim 4.8$. Thus $r=\textrm{max}\lbrace \lbrace \frac{k}{2}   \rbrace \cup \lbrace \rho(j,V_i) \mid V_i
\textrm{ vertex of } \F \rbrace=\sim 3.33$ and $\tilde{N}=2cosh(2\cdot r)=\sim 780.6$. As $2+15 \cdot
52=782$, we have to compute $\F'_n$ for $5 \leq n \leq \frac{52}{2}=26$.
However we find that all $\F'_n$, for $5 \leq n \leq 26$, are empty and hence
$S_{15}$ is a generating set for the whole group. All this and the case $d=23$ is summed up in the following theorem.

Note that in the next theorem $\Psi^{-1}(\Psi(\Gamma)_0) = \langle i,j\rangle$.

\begin{theorem}
\label{onecase}
Let $\Gamma
=\SL_{1}(\h(-1,\,-1,\Z [\frac{1+\sqrt{-d}}{2}]))$ and let $Y_d$ be a finite set of units $\gamma$ such that $B(\gamma) \in \F_N$ and let $S_d=\Psi^{-1}(\Psi(\Gamma)_0)\cup \{g(\gamma) \mid g\in \langle \sigma , \tau
\rangle ,\ \gamma \in Y_d\}$.

\begin{enumerate}
\item If $\ d=15$ then $\Gamma =\langle S_d\rangle $, where
\begin{eqnarray*}
Y_{15} & = & \left\lbrace 2+(-1+\omega)i+2j+\omega k,2+2i+(1-\omega)j+ \omega k,-2+(-2+\omega)i+(1+\omega)j \right.\\
&& -2+(-1+\omega)i+\omega j -2k,(-2+\omega)+(1+\omega)i+2j,\\
&& (-4+2\omega)+2i+3j+(-2-2\omega)k,(-4+2\omega)+3i+2j+(-2-2\omega)k,\\
&& \left. (-4+2\omega)+3i-2j+(-2-2\omega)k,(-4+2\omega)+(2+2\omega)i+3j-2k \right\rbrace.
\end{eqnarray*}
\item If $\ d=23$ then $\langle S_d\rangle $ has finite index in $\Gamma$, where
\begin{eqnarray*}
Y_{23}&=& \left\lbrace(-3+\omega )+(2+\omega )i,(-2+\omega )-2i+(1+\omega )j-2k, \right. \\
&&(-3+\omega )-(2+\omega )k,-2-2i+(-2+\omega )j-(1+\omega )k,-2+(-2+\omega )i+(1+\omega )j-2k, \\
&& -4+2i-2j+(1-2\omega)k,-4+(-1+2\omega)i+2j+2k, (-4+2)-2i+(1+\omega)j-(4+\omega)k,\\
&&(-5+\omega)+(2-\omega)i-2j-(2+2\omega)k,(-4+2\omega)+(1+\omega)i+2j-(4+\omega)k, \\
&&(-4-\omega)+(1+\omega)i-2j+(4-2\omega)k,(-5+\omega)+2i+(-2+\omega)j-(2+2\omega)k,\\
&& (-4-\omega)+(-4+2\omega)i+(1+\omega)j+2k,(-5+\omega)+(2+2\omega)k+(2-\omega)j+2k,\\
&&-6-3i+(-2+2\omega)j-2\omega k, -5+(2+2\omega)i+2j+(4-2\omega)k,-6+2\omega i+3j+(2-2\omega)k, \\
&&-5-2i+(-4+2\omega)j-(2+2\omega)k, -5+(-4+2\omega)i+(2+2\omega)j-2k,\\
 && \left. -6+2\omega i -3j+(2-2 \omega)k \right\rbrace.
\end{eqnarray*}
\end{enumerate}
\end{theorem}

Part (a) of the figures below
shows the projection on $\partial \HQ^3$ of  the bisectors of the elements in $Y_d$
 and part (b) shows the fundamental domain of the subgroup $\langle S_d\rangle$.

\begin{figure}[H]
\centering
\subfigure[Projection of bisectors in $Y_{15}$.]{
\includegraphics[scale=0.5]{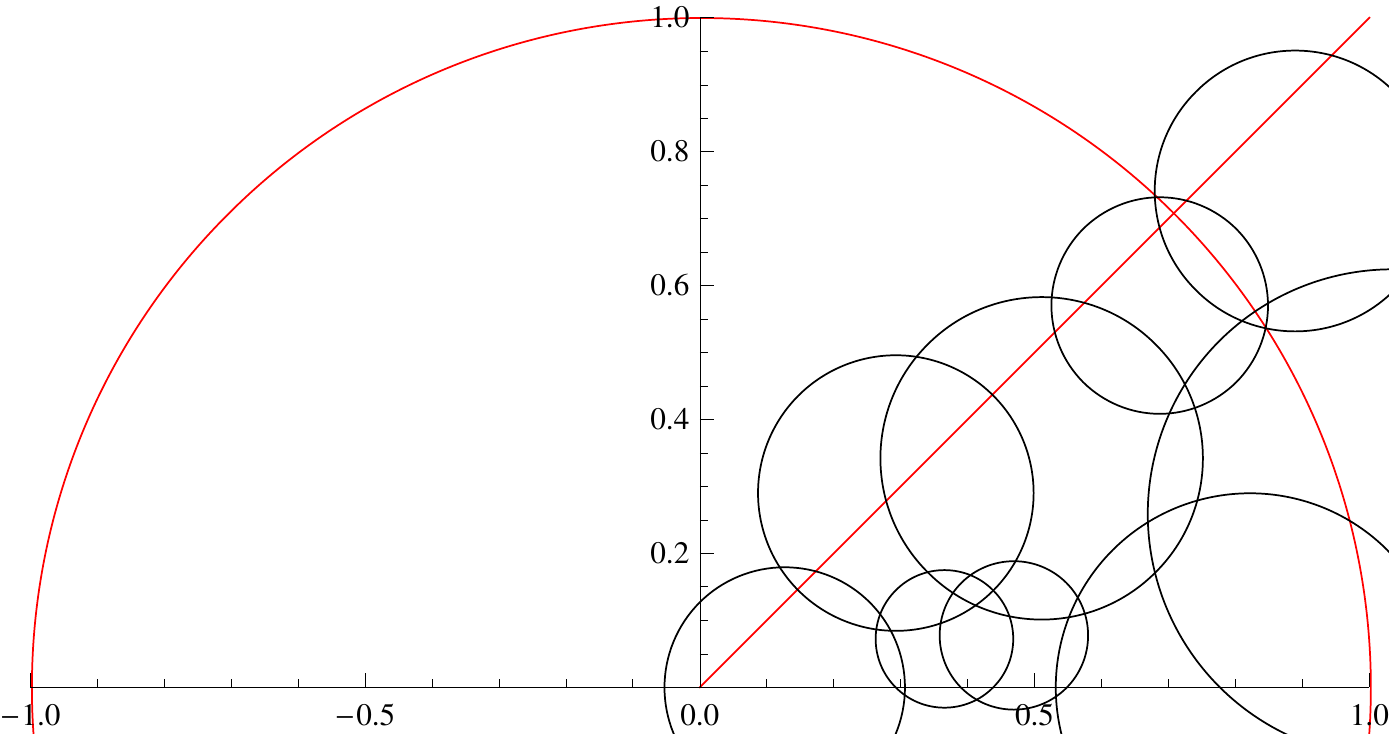}
\label{d152dim}
}
\subfigure[Fundamental domain of $\ \langle S_{15}\rangle$. ]
{\includegraphics[scale=0.5]{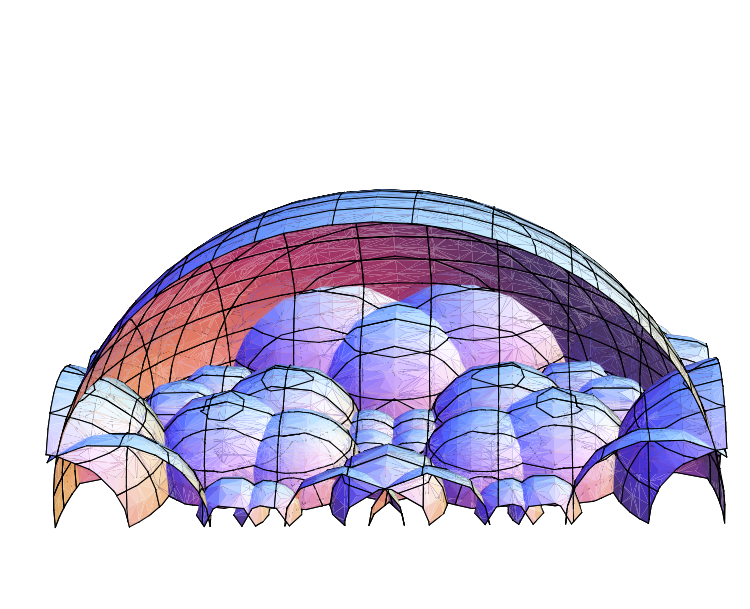}
\label{d15}
}
\caption{}
\end{figure}

\begin{figure}[H]
\centering
\subfigure[Projection of bisectors in $Y_{23}$.]{
\includegraphics[scale=0.5]{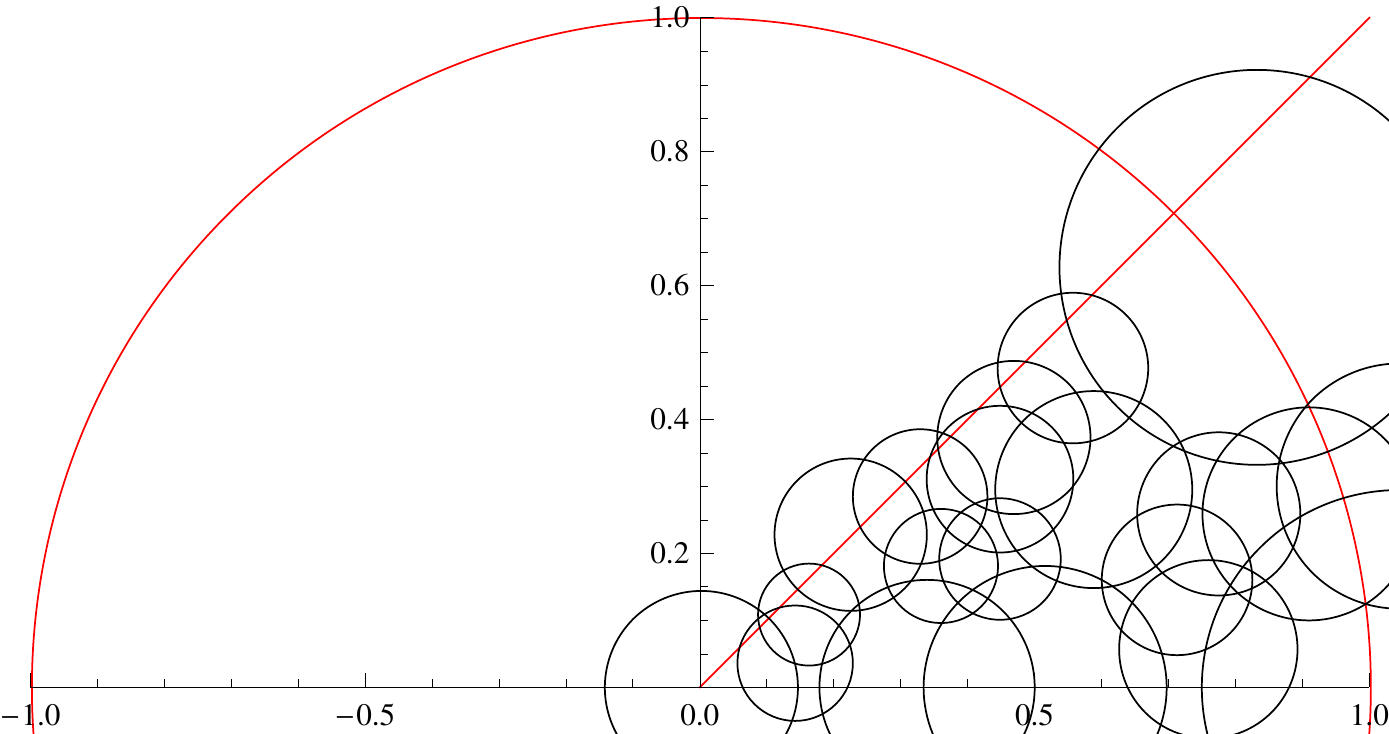}
\label{d232dim}
}
\subfigure[Fundamental domain of $\ \langle S_{23}\rangle$.]{
\includegraphics[scale=0.5]{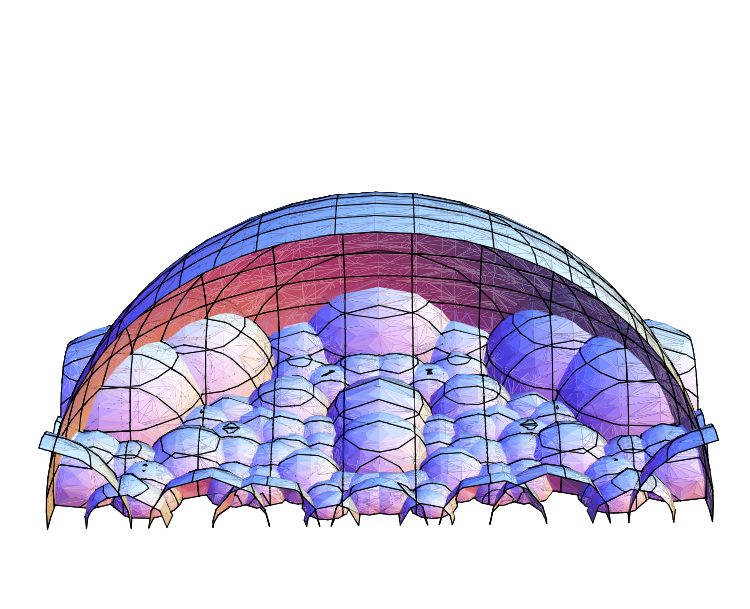}
\label{d23}
}
\caption{}
\end{figure}

\begin{remark}
In the context of  a division algebra, Pell and Gauss units were constructed in \cite{jpsf, jurcal}.  For some of these units $\gamma$, say, we have that  $P_{\gamma}=0$ and these are typically associated to hyperbolic elements.   Together, these units seem to generate a subgroup of finite index. If so, these can be used as generic units in the group ring problem. Still in the same context, in \cite{jpsf} it was proved that $SL_{1}(\h(-1,\,-1,\mathfrak{o}_K))$ is a hyperbolic group for all positive $d \equiv 7 \pmod 8$. This line of classification was introduced and initiated in \cite{jpp}.
\end{remark}

The next cases we treat  are  division algebras of  the form $\h (a,b,\Q(i')), (i'=\sqrt{-1})$, with $ 0 < a < b$ integers  and $\Gamma= \SL_1(\h(a,b,\mathbb{Z}[i']))$. For $a=2$ and $b=5$, this  can be found in \cite[Section X]{elstrodt}, and it is commensurable with $\mbox{PO}_4(\Z, q)$, where $q(x,y,z,w)=-2x^2-5y^2-10z^2+w^2$.

Let $u\in \Gamma$ and write  $u= u_0+u_1i+u_2j+u_3k$, with $u_t \in \mathbb{Z}[i']$. In this case $ \gamma_u=$\\
$ \begin{pmatrix}\label{mat}
u_0+u_1\sqrt{a}&u_2\sqrt{b}+u_3\sqrt{ab} \\
u_2\sqrt{b}-u_3\sqrt{ab}&u_0-u_1\sqrt{a}
\end{pmatrix}$, as stated in section \ref{secback}. Let  $u_t=x_t+y_ti'$, $x=(x_0,x_1,x_2,x_3)$, $y=(y_0,y_1,y_2,y_3)$, $q(x) =x_0^2-ax_1^2-bx_2^2+abx_3^2$ and let $B(x,y)$ be the bilinear form associated to $q(x)$. The next lemma describes $\Gamma$ as a subgroup of $\PSL_2(\C )$ in terms of a system of Diophantine equations.

\begin{lemma}\label{unitshabi}
Let $u=u_0+u_1i+u_2j+u_3k \in \SL_1(\h(a,b,\mathbb{Z}[i']))$, and let $\gamma_u$ and the vectors $(x_0,x_1,x_2,x_3)$ and $(y_0,y_1,y_2,y_3)$ be as described above. Then
\begin{eqnarray*}
\begin{cases}
q(x)-q(y)=1\\
B(x,y)=0
\end{cases}
\end{eqnarray*}
Moreover if we set $x_0^2+ay_1^2+by_2^2+abx_3^2=n\in \N$, we obtain the following system.
\begin{eqnarray}\label{system}
\begin{cases}
x_0^2+ay_1^2+by_2^2+abx_3^2=n\\
y_0^2+ax_1^2+bx_2^2+aby_3^2=n-1\\
B(x,y)=0\\
\Vert \gamma_u \Vert^2=4n-2
\end{cases}\end{eqnarray}
\end{lemma}

\begin{proof}
The first set of equations follows from the fact that the determinant of $\gamma_u$ is $1$, $\Gamma$ being a subgroup of $\PSL_2(\C)$. For the three first equations of the second set we replace $x_0^2+ay_1^2+by_2^2+abx_3^2$ by $n$ in the first set. The last equation is just mere calculation.
\end{proof}

Here  $\Gamma_j=1$ and so we are in Case I. Taking $a=2,b=5$ and $n=2$ in system (\ref{system}), we find $u$ such that  $ \gamma_u=\begin{pmatrix}
i'-i'\sqrt{2} & 0 \\
0 & i'+i'\sqrt{2}
\end{pmatrix}$ and thus  $\vert a \vert^2 + \vert c \vert^2 = 3-2\sqrt{2}<1$ and hence the bisector associated to this unit gives a starting point to run the DAFC (as explained in Case I). Moreover Proposition \ref{unitshabi} shows that the sequence $r_n$ of Proposition \ref{DAFC} may be taken as $r_n=4n-2$. Using Mathematica, we compute $N=102$, which means that the units needed in a generating set all have norm smaller than $4\cdot 102-2$. In this case the fundamental domain contains again the symmetries given by the group $\langle \sigma ,\tau \rangle$ described in Proposition \ref{symmetries}. Moreover the symmetry $\phi$ from Proposition \ref{symmetries} may also be taken into account. Indeed, as stated in the proof of Proposition \ref{symmetries}, in $\BQ^3$ the map $\Psi(\phi)$ is the reflection $\pi$ in the plane $\lbrace (x,y,z) \in \R^3 \mid z=0 \rbrace$. As $\Psi(\Gamma)_0$ is trivial here, we have to cover the whole $\partial\BQ^3$ (recall that if $\Psi(\Gamma)_0$ is not trivial, we only have to cover the part $\partial \F_0 \cap \partial \BQ^3$) and hence the symmetry $\pi$ and thus also $\phi$ may be taken into account.
Thus the group $\langle \sigma ,\tau , \phi \rangle$ is contained in the group of symmetries of $\Gamma$.

\begin{theorem}
 In $\SL_1(\h(2,5,\mathbb{Z}[i']))$, the  subgroup $S_{2,5}(\Z [i]) = \langle  -1,g(\gamma) \mid g\in \langle \sigma^2,\tau , \phi \rangle  ,\gamma \in Y \rangle $, where $Y$ is given below,  has finite index.

\begin{eqnarray*}
Y &=& \left\lbrace i'-i'i, 2i'-i'j, 2-i'i-j, 2-2i+i'j, 3i'+k, 3+i'i-i'k, 2+i'i-i'j-i'k, \right. \\
&&4i'-i'i-i'j+k, 3-3i+k, 2-3i+i'j+k, 3-i'i-2j-k,6-i'j-2i'k, \\
&&2+4i'i-i'j-2i'k,6-3j-k, 2-4i'i-3j-k, (1-4i')+(4+2i')i+(2+2i')j+2k, \\ &&(1+4i')+(4-2i')i+(-2-2i')j-2k,3-7i+4i'j+k, \\
&&\left. (8+3i')+(-4+2i')i+(-2-2i')j+(1-2i')k \right\rbrace .
\end{eqnarray*}

\end{theorem}

Figures \ref{25qi2dim} and \ref{25qi} show a part of the fundamental domain  and  the projection on $\partial \HQ^3$ of the fundamental domain. The bisectors coming from the elements in $Y$ are drawn in bold face.

\begin{figure}[H]
\centering
\subfigure[Projection of the fundamental domain of $S_{2,5}(\Z \lbrack i \rbrack )$.]{
\includegraphics[scale=0.5]{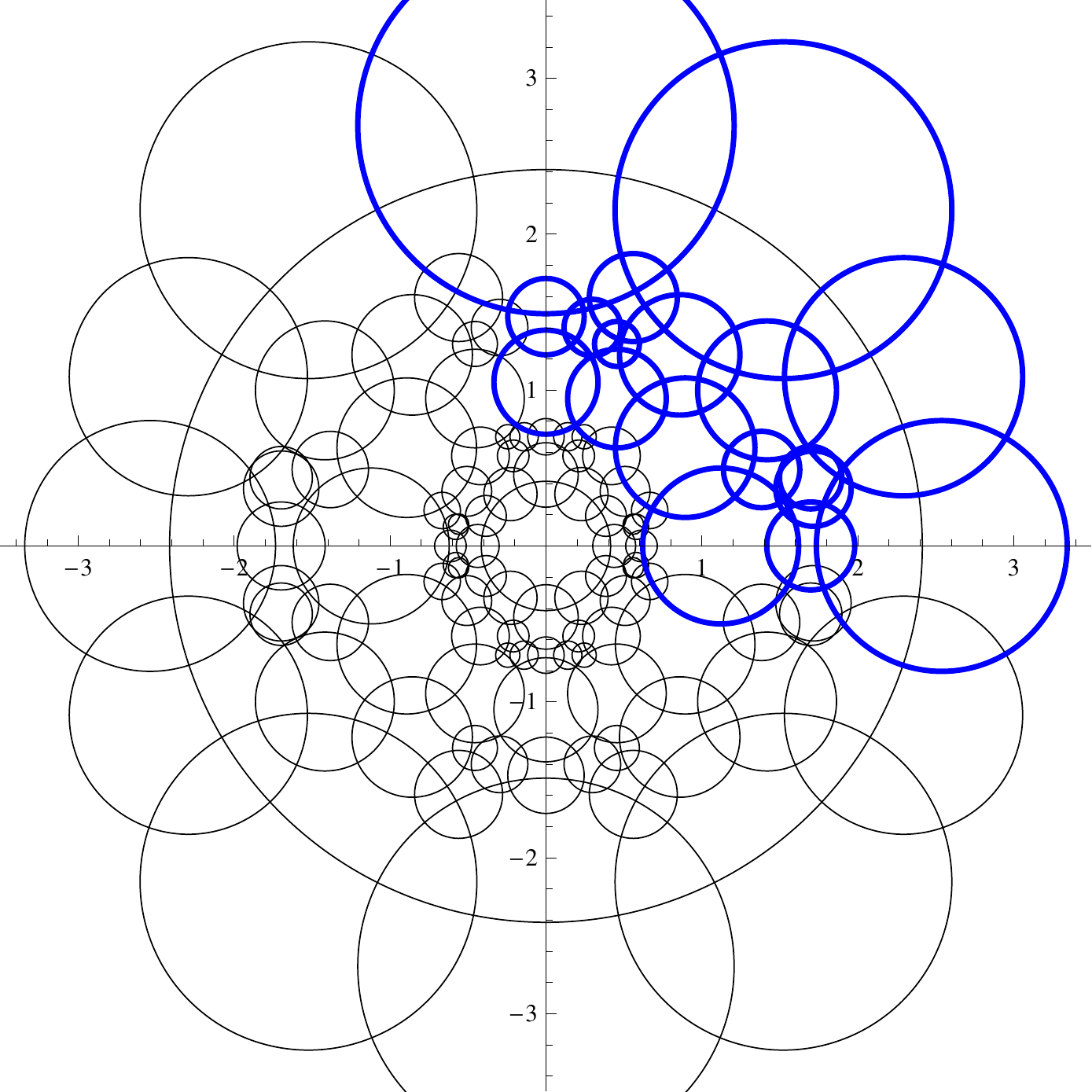}
\label{25qi2dim}
}
\subfigure[Part of the fundamental domain of $S_{2,5}(\Z \lbrack i \rbrack )$. ]{
\includegraphics[scale=0.5]{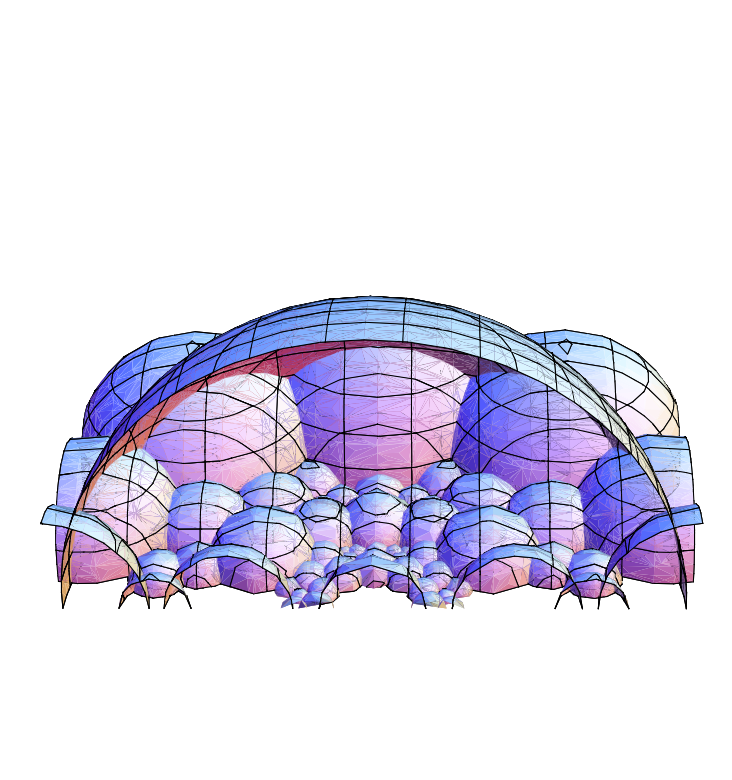}
\label{25qi}
}
\caption{}
\end{figure}

We finish this section with an example of orders in division algebras over $\Q$, i.e., with Fuchsian groups.
Consider a division algebra of the type $\h(a,b,\mathbb{Q})$ with $a>0, b>0$ integers (for example with $a=2$ and $b=5$). We consider the following order $\h(a,b,\mathbb{Z})$.
For $u=x_0+x_1i+x_2j+x_3k \in \h(a,b,\mathbb{Z})$   we have
$ \gamma_u =
\begin{pmatrix}
x_0 + x_1\sqrt{a} & x_2\sqrt{b} + x_3 \sqrt{ab} \\
x_2\sqrt{b} - x_3 \sqrt{ab} & x_0 - x_1\sqrt{a}
\end{pmatrix}.$  Letting $x=(x_0,x_1,x_2,x_3)$, $q_1(x)=x_0^2-ax_1^2-bx_2^2+abx_3^2$ and $q_2(x)= x_0^2+ax_1^2+bx_2^2+abx_3^2$, we obtain the following lemma.

\begin{lemma}
\begin{eqnarray} \begin{cases}
q_1(x)=1 \\
2q_2(x) =\|\gamma\|^2
\end{cases} \end{eqnarray}
Moreover if $q_2(x_0,0,0,x_3)=n$ we obtain
\begin{eqnarray} \begin{cases}
x_0^2+abx_3^2=n \\
ax_1^2+bx_2^2=n-1\\
\| \gamma\|^2= 4n-2, n\in \mathbb{N}
\end{cases} \end{eqnarray}
\end{lemma}

This lemma being similar to Lemma \ref{unitshabi}, we omit the proof. Taking $a=2, b=5$ we have that $ n=1 \ \mbox{or}\  n\geq 9$. Clearly  the stabilizer of $i$, $\Gamma_i$, is trivial (note that as we are in dimension $2$ the role of $j$ is played by $i$.) For $n=1$ we find the identity matrix and for $n=9$ we find an element whose bisector $\Sigma$ separates $i$ and $0$ in the ball model, which gives a starting point for the DAFC as described in Case I. The sequence $r_n$ needed for the DAFC is defined here by $r_n=4n-2$ and $N=46$. Hence we obtain a set of generators for a subgroup of finite index containing units of norm smaller than $4\cdot 46-2$. Figure \ref{H25q} shows the resulting fundamental domain.

\begin{figure}[H]
\centering
\includegraphics[scale=0.50]{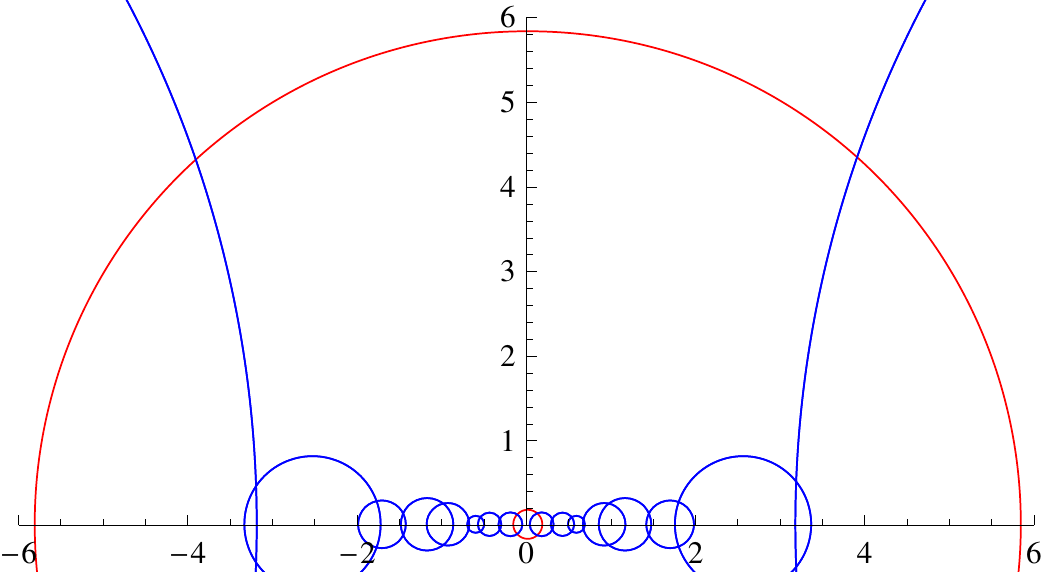}
\caption{Fundamental domain of $S_{2,5}(\Z )$.}
\label{H25q}
\end{figure}

\begin{theorem}
The group  $S_{2,5}(\Z )=\langle -1, g(\gamma) \mid g\in \langle   \sigma^2 , \phi\rangle ,\gamma \in Y\}$, where $Y = \lbrace 3+2i, 2+2i+j+k, 3+3i+k, 6+3j+k \rbrace$, is a subgroup of finite index of $\SL_1( \h(2,5,\mathbb{Z}))$.
\end{theorem}


\subsection{Matrix Algebras}

In this section we consider groups $\Gamma$ of finite covolume and with  at least one ideal vertex (so $\Gamma$ is not cocompact). First we remind that a horosphere $\Sigma$ of $\BQ^3$, based at a point $P$ of $S^2$, is the intersection with $\BQ^3$ of a Euclidean sphere in $\overline{\BQ^3}$ tangent to $S^2$ in $P$. The interior of a horosphere is called a horoball. Recall then that an ideal vertex of a convex polyhedron $\F$ in $\BQ^3$ is a point $P$ of $\F \cap \partial \BQ^3$ for which there is a closed horoball $C$ of $\overline{\BQ^3}$ based at $P$ such that $C$ meets just the sides of $\F$ incident to $P$ and such that $\partial C \cap \F$ is compact. More details on this may be found in \cite[Chapter 6.4]{ratcliffe}. To describe how to implement the DAFC and obtain a fundamental domain we work in $\BQ^3$ and suppose that $j$ is an ideal vertex. Recall that $S^2=\partial \BQ^3$.
\vp

\noindent {\bf{Case I: $\Psi(\Gamma)_0$ is trivial.}} As in subsection \ref{divalg}, in this case again, we may just run the DAFC as stated in Proposition \ref{DAFC} (respectively \ref{completeset}). However again because of visualisation in the upper half space model, we are first looking for some "starting point". So find a $\gamma = \left(
\begin{array}{ll}
a & b \\ c & d
\end{array}
\right)\in \Gamma$, with $|a|^2+|c|^2=1$. By Lemma \ref{controlfundom}, these are exactly the units associated to bisectors that
contain the point $j$ (note that since $\Psi(\Gamma)_{0}$ is trivial we have that $\gamma \not\in \SU_{2}(\C)$). Since $\Gamma$ has finite covolume, a finite number of them, those with smallest matrix norm, $\gamma_1,\cdots ,
\gamma_m$ say, can be chosen such that  there exists a neighbour $V_j$ of $j$ in $S^2$ such that $(V_j \setminus j) \subseteq  \bigcup\limits_{1\leq k\leq m}\mbox{Interior}(\Sigma_{\Psi (\gamma_k)})$. To know the value of $m$ one just has to know the link of $j$. Recall that
the link of the ideal vertex $j$ is defined to be the set $\Sigma_j \cap \F$, where $\Sigma_j$ is a horosphere based at $j$ that meets just the sides of $\F$ incident with $j$. Set
$$V=(S^2 \setminus \lbrace j \rbrace) \cap \bigcap\limits_{1\leq k\leq m} \mbox{Exterior}(\Sigma_{\Psi (\gamma_k)}).$$
Since $\Gamma$ is of finite covolume, we may use the DAFC to find $N \in \N$ such that $V$ is contained in $\F_N$.  The   output is $\langle \gamma_1,\cdots \gamma_n\rangle$, a subgroup of finite index in $\Gamma$, where $\{\gamma_{m+1}.\cdots \gamma_n\}$ is the set of all $\gamma_i$ such that $B(\gamma_i) \in \F_N$.
\vp

\noindent {\bf{Case II: $\Psi(\Gamma)_0$  non-trivial.}} In this case, proceed as in Case I with $$V=(\partial\mathcal{F}_0)\cap (S^2 \setminus \{j\})\bigcap\limits_{1\leq k\leq m}\mbox{Exterior}(\Sigma_{\Psi (\gamma_k)}),$$
where $\mathcal{F}_0$ is a fundamental domain of $\Psi(\Gamma)_0$. Then $\langle G_0,\gamma_1,\cdots \gamma_n\rangle$ is of finite index in $\Gamma$, where $G_0$ is a generating set of $\Psi(\Gamma)_0$.
\vp

The examples we give here are the Bianchi Groups, i.e. the groups $\PSL_2(\O_d)$ with $\O_d$ the ring of integers in
$\Q(\sqrt{-d})$ and $d>0$, (see \cite{elstrodt, floge}), for which a Ford Fundamental region is given  in \cite[Theorem
VII.3.4]{elstrodt}. Recall that a Ford fundamental region is defined in the following way. Let $\Gamma$ be a
group acting discontinuously on $\HQ^3$, such that no non-trivial element of $\Gamma$ fixes the point $\infty \in \partial \HQ^3$.
Recall that $\Iso_{\gamma}$ denotes the isometric sphere associated to any non-trivial $\gamma \in \Gamma$ and  denote the exterior of $\Iso_{\gamma}$ by $H_{\gamma}$. Then it may be shown that $\bigcap_{\gamma \in \Gamma, \gamma \neq 1}
H_{\gamma}$ is a fundamental region (i.e. a fundamental domain which is not necessary connected). This is called the
Ford fundamental region. We are not going into further details on this topic, but the interested reader we refer to
\cite[Chapter 9.5]{beardon}. Here we  describe a Dirichlet fundamental polyhedron for all $d$. Note that the Bianchi
groups can also be handled as groups commensurable with  the unit group of an order in the split quaternion
algebra  $\h(K)$,  $K
=\Q (\sqrt{-d})$  and $d\equiv 1,2 \mbox{ mod }4$ or $d\equiv 3 \textrm{ mod } 4$ and $\h(K)$ not a division ring. All this can be handled as in the previous section (the division assumption in the previous section was only used to guarantee that the groups were cocompact and hence of finite covolume, see \cite[Theorem X.1.2]{elstrodt}).

Let $\gamma =\begin{pmatrix}
a && b\\
c && d
\end{pmatrix} \in \Gamma =\PSL_2(\O_d)$. Let $\omega=\sqrt{-d}$ if $d \equiv 1,2 \  mod\  4$ and $\omega =\frac{1+\sqrt{-d}}{2}$ if $d \equiv 3\  mod\  4$. $\O_d$ is defined as $\Z[\omega]$. Note that $\Gamma$ is not cocompact. This follows from \cite[Theorem VII.1.1]{elstrodt} 
We have that $\Psi^{-1}(\Psi(\Gamma)_0) = \Gamma \cap \SU_2(\C)$. Recall from section \ref{seciso} that $\gamma \in \SU_2(\C)$ if and only if $\Vert \gamma \Vert^2=2$. Thus in this case $\vert a(\gamma) \vert^2+ \vert b(\gamma) \vert^2 + \vert c(\gamma) \vert^2 + \vert d(\gamma) \vert^2=2$ and $a(\gamma)d(\gamma)-b(\gamma)c(\gamma)=1$. If $d=2$ or $d>3$, then for every element $a \in \O_d$, $\vert a \vert^2 >1$ except if $a= \pm 1$ or $a=0$ and thus the only element in $\Gamma \cap \SU_2(\C)$ in those cases is $\gamma_0=\begin{pmatrix} 0 && -1\\ 1 && 0\end{pmatrix}$. So $\Psi(\Gamma)_0 = \langle \Psi(\gamma_0)\rangle$. However if $d=1$, one easily computes that $\Psi(\Gamma)_0 = \langle \Psi (\gamma_1), \Psi(\gamma_2) \rangle \cong C_2 \times C_2$, where $\gamma_1=\begin{pmatrix} i & 0 \\ 0 & -i \end{pmatrix}$ and $\gamma_2=\begin{pmatrix} 0 & i \\ i & 0 \end{pmatrix}$. If $d=3$, $\Psi(\Gamma)_0 = \langle \Psi (\gamma_1), \Psi(\gamma_2) \rangle \cong S_3$, where $\gamma_1=\begin{pmatrix} \omega & 0 \\ 0 & \overline{\omega} \end{pmatrix}$ and $\gamma_2=\begin{pmatrix} 0 & -1 \\ 1 & 0 \end{pmatrix}$. Since $\Psi(\gamma_0) \in \Psi(\Gamma)_0,$ for all $d$, we have that a fundamental domain $\mathcal{F}_0$ of $\Gamma$ is a subset of $\{(x,y,z)\in \BQ^3 \mid z \geq 0 \}$ in the ball model and in the upper half unit sphere in the upper half space model. Since $j$ is an ideal vertex of $\Gamma$, we have to find the elements $\gamma$ such that $j\in \Sigma_{\Psi (\gamma )}$, as it is explained in Case I. By Lemma \ref{controlfundom}, this is the case if and only if $\vert a \vert^2 + \vert c \vert^2=1$. As $\vert a \vert \geq 1$ for every $0 \neq a \in \O_d$, one of $a$ or $c$ has to be $0$. We may suppose that $c (\gamma )=0$. Indeed if $c (\gamma )\neq 0$ then $a(\gamma_0\gamma )\neq 0$, and hence $c(\gamma_0\gamma)=0$. As $\Sigma_{\Psi (\gamma_0\gamma )}=\Sigma_{\Psi (\gamma )}$, we can hence suppose $c(\gamma)$ to be $0$. If $c(\gamma)=0$, then $\gamma$ fixes the point $\infty \in \partial \HQ^3$ and hence $\Psi(\gamma)\in \Psi(\Gamma)_{j}$ in $\BQ^3$. Denoting  by $\mathcal{F}_j$ a fundamental domain of $\Psi(\Gamma)_j$ acting on $\BQ^3$, we have that $\mathcal{F} \subseteq \mathcal{F}_0\cap \mathcal{F}_j\cap \{(x,y,z)\in \BQ\ | z\geq 0\}$. So, referring to Case II, we may take $\{\Psi(\gamma_1),\cdots , \Psi(\gamma_m)\}\subseteq \Psi(\Gamma)_j$ and $V= (\partial \mathcal{F}_j)\cap (\partial \mathcal{F}_0)\cap (S^2 \setminus \lbrace j \rbrace)$.

Using $\eta_0: \HQ^3 \rightarrow \BQ^3$ we transfer this information to $\HQ^3$. In this model $\Gamma_{\infty}=\{ \begin{pmatrix}
a && b\\
0 && d
\end{pmatrix}  \ | \ ad=1, a,b,d\in \Z [\omega] \}$.
Let $\hat{\mathcal{F}}_{\infty}$ be a fundamental domain of  $\Gamma_{\infty}$ acting on $\C$. Then,  for all $d$, $\mathcal{F}_{\infty}=\{ z+rj\in \HQ^3\ |\ z\in \hat{\mathcal{F}}_{\infty}\}$ is a fundamental domain of $\Gamma_{\infty}$  in $\HQ^3$.
The following lemma gives more details on $\F_{\infty}$. As it is proved by easy computations we omit the proof.

\begin{lemma}\label{startingpoint}
\begin{enumerate}
\item If $1<d \equiv 1,2 \  mod\  4$ then  $\mathcal{F}_{\infty} = \lbrace z + rj \in \HQ^3 \mid -\frac{1}{2} \leq Re(z) \leq \frac{1}{2},\  -\frac{\sqrt{d}}{2} \leq Im(z) \leq \frac{\sqrt{d}}{2} \rbrace$
\item If $3<d \equiv 3\  mod\  4$ then  $\mathcal{F}_{\infty} = \lbrace z + rj \in \HQ^3 \mid -\frac{1}{2} \leq Re(z) \leq \frac{1}{2},\ -\frac{1+d}{4} \leq Re(z)+\sqrt{d}Im(z) \leq \frac{1+d}{4},\ -\frac{1+d}{4} \leq Re(z)-\sqrt{d}Im(z) \leq \frac{1+d}{4} \rbrace$
\item If $d=3$, then $\mathcal{F}_{\infty} = \lbrace z + rj \in \HQ^3 \mid 0 \leq Re(z) \leq \frac{1}{2},\ 0 \leq Re(z)+\sqrt{3}Im(z) \leq 1 \rbrace$.
\item If $d=1$ then $\mathcal{F}_{\infty} = \lbrace z + rj \in \HQ^3 \mid -\frac{1}{2} \leq Re(z) \leq \frac{1}{2},\  0 \leq Im(z) \leq \frac{1}{2} \rbrace$.
\end{enumerate}
\end{lemma}

For $d \equiv 1,2 \  \mbox{mod}\  4$,  $\hat{\mathcal{F}}_{\infty}$ is a rectangle with vertices $\pm \frac{1}{2}\pm \frac{\sqrt{d}}{2}i$ and for  $d \equiv 3\  mod\  4$ it is a hexagon with vertices $\pm\frac{ (d+1)\sqrt{d}}{4d}i$ and $\pm \frac{1}{2}\pm \frac{(d-1)\sqrt{d}}{4d}i$. For $d\neq 3$, all vertices of this hexagon lie on the circle centered at  $0$ with radius $\frac{(d+1)\sqrt{d}}{4d}$. Hence, for $d\in \{1,2,3,7,11\}$, $\F_{\infty} \cap \partial \HQ^3$ is included in the interior of $S^2 \cap \partial \HQ^3$ and hence $\infty$ is the only ideal vertex of $\Gamma$ in $\widehat{\HQ}^3$, respectively $j$ is the only ideal vertex in $\widehat{\BQ}^3$.

We implemented the DAFC for some Bianchi groups for $d\equiv 3 \mod 4$. Note that the implementation for Bianchi groups for $d \equiv 1,2 \mod 4$ is done in the same way, the only difference lies in the definition of $\omega$ and $\F_{\infty}$.
So let $\Gamma=\PSL_2(\Z[\omega])$ where $\omega=\frac{1+\sqrt{-d}}{2}$ for $d\equiv 3 \mod 4$. The following lemma describes the group $\Gamma$ in terms of Diophantine equations.
\begin{lemma}
Let $\gamma=\begin{pmatrix} a && b \\ c && d \end{pmatrix} \in \Gamma$, $a=x_0+y_0\omega$, $b=x_1+y_1\omega$, $c=x_2+y_2\omega$, $d=x_3+y_3\omega$, $x=(x_0,x_1,x_2,x_3)\in \Z^{4}$, $y=(y_0,y_1,y_2,y_3)\in \Z^{4}$,  $\det (x) = x_0x_3-x_1x_2$, $J(x)=(x_3,-x_2,-x_1,x_0)$ and $\det(y)$ and $J(y)$ are analogously defined as $det(x)$ and $J(x)$.
Then
\begin{eqnarray}
\begin{cases}
\det (x)-(\frac{d+1}{4})\det (y)=1\\
\det (y)+\langle x|J(y)\rangle =0\\
\|\gamma\|^2=\|x\|^2+(\frac{d+1}{4})\|y\|^2+\langle x|y\rangle\in \N.
\end{cases}
\end{eqnarray}
\end{lemma}

\begin{proof}
If we compute the determinant of $\gamma$, we get that
$$det(x)-\frac{d-1}{4}det(y)+\frac{1}{2}\langle x|J(y)\rangle + (det(y)+\langle x|J(y)\rangle)\frac{\sqrt{-d}}{2}=1.$$
Hence we have that
$$det(x)-\frac{d-1}{4}det(y)+\frac{1}{2}\langle x|J(y)\rangle=1$$
and
$$det(y)+\langle x|J(y)\rangle=0.$$
Replacing $\langle x|J(y)\rangle$ by $-det(y)$ in the first equation, we get the two first equations of the lemma. The third equation comes from mere computations of $\Vert \gamma \Vert^2$.
\end{proof}

If $\|\gamma\|^2=n\in \N$ then $\mbox{max}\{|a|,|b|,|c|,|d|\}\leq \sqrt{n}$  and hence the equation $\|\gamma\|^2=n$ has a finite number of solutions $(x,y)$. Consequently, for each $n\in \N$ the system above has a finite number of solutions $(x,y)$ such that $\|\gamma\|^2=n$. To implement the DAFC we choose the sequence $r_n=n$, for  $n \geq 1$, because $\Vert \gamma \Vert^2 \in \N$. The next theorem gives the outcome of the DAFC for three  examples.

\begin{theorem}

For $d\in \{19,23,27\}$, the subgroup  $\Gamma_{(d)}=\langle \Psi^{-1}(\Psi(\Gamma)_0), g(Y_d)\ |\ g \in \langle \sigma^2 , \tau
\rangle \rangle$, where $Y_d$ is given below, is of finite index in $\PSL_2(\mathbb{Z}[\omega])$, where  $\omega=\frac{1+\sqrt{-d}}{2}$.

$ Y_{19}= \left\lbrace
\begin{pmatrix}
1 & 1\\
0 & 1
\end{pmatrix},
\begin{pmatrix}
1 & -\omega\\
0 & 1
\end{pmatrix},
\begin{pmatrix}
1-\omega & 2\\
2 & \omega
\end{pmatrix}\right\rbrace,
$

\vp

$ Y_{23}= \left\lbrace
\begin{pmatrix}
1 & 1\\
0 & 1
\end{pmatrix},
\begin{pmatrix}
1 & -\omega\\
0 & 1
\end{pmatrix},
\begin{pmatrix}
-2+\omega & 3\\
-1-\omega & -3
\end{pmatrix},
\begin{pmatrix}
-3+\omega & 2+\omega\\
-2-\omega & -3+\omega
\end{pmatrix} \right\rbrace,
$

\vp

$ Y_{27}= \left\lbrace
\begin{pmatrix}
1 & 1\\
0 & 1
\end{pmatrix},
\begin{pmatrix}
1 & -\omega\\
0 & 1
\end{pmatrix},
\begin{pmatrix}
2 & -\omega\\
1-\omega & -3
\end{pmatrix}\right\rbrace,
$
\vp

\end{theorem}

The pictures below show a fundamental domain, rotated over ninety degrees, of $\Gamma_d$ and its projection on $\partial \HQ^3$ for $d\in\{19,23,27\}$,  respectively.

\begin{figure}[H]
\centering
\subfigure[$\Gamma_{(19)}$]{
\includegraphics[scale=0.2]{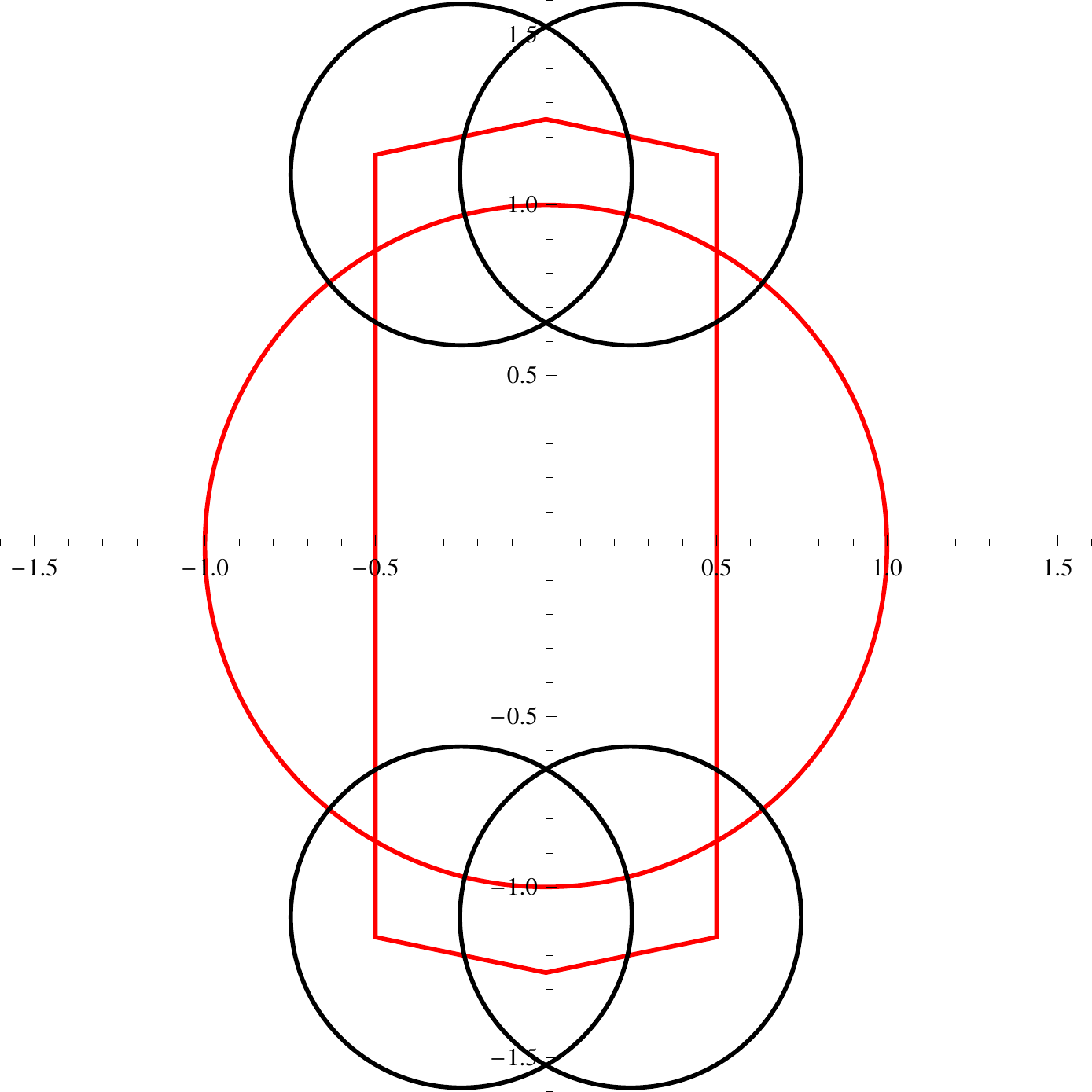}
\label{psl192dim}
}
\subfigure[$\Gamma_{(19)}$]{
\includegraphics[scale=0.3]{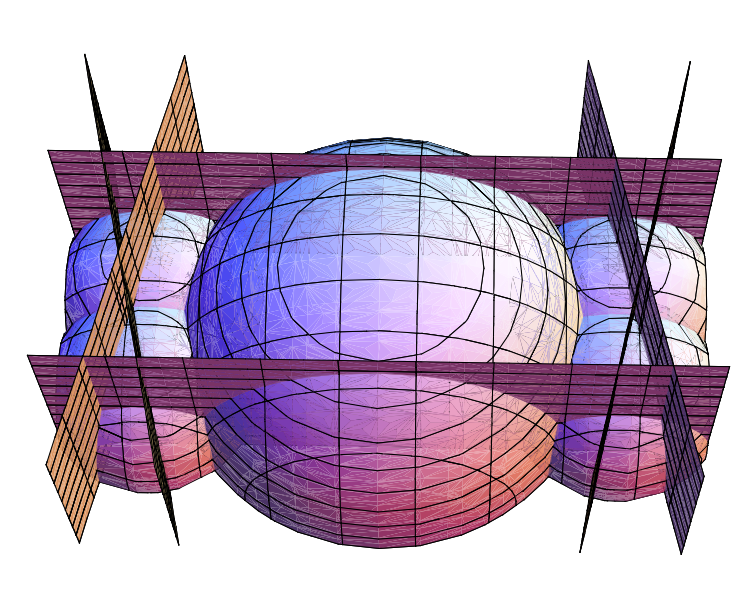}
\label{psl19}
}
\caption{}
\end{figure}

\begin{figure}[H]
\centering
\subfigure[$\Gamma_{(23)}$]{
\includegraphics[scale=0.2]{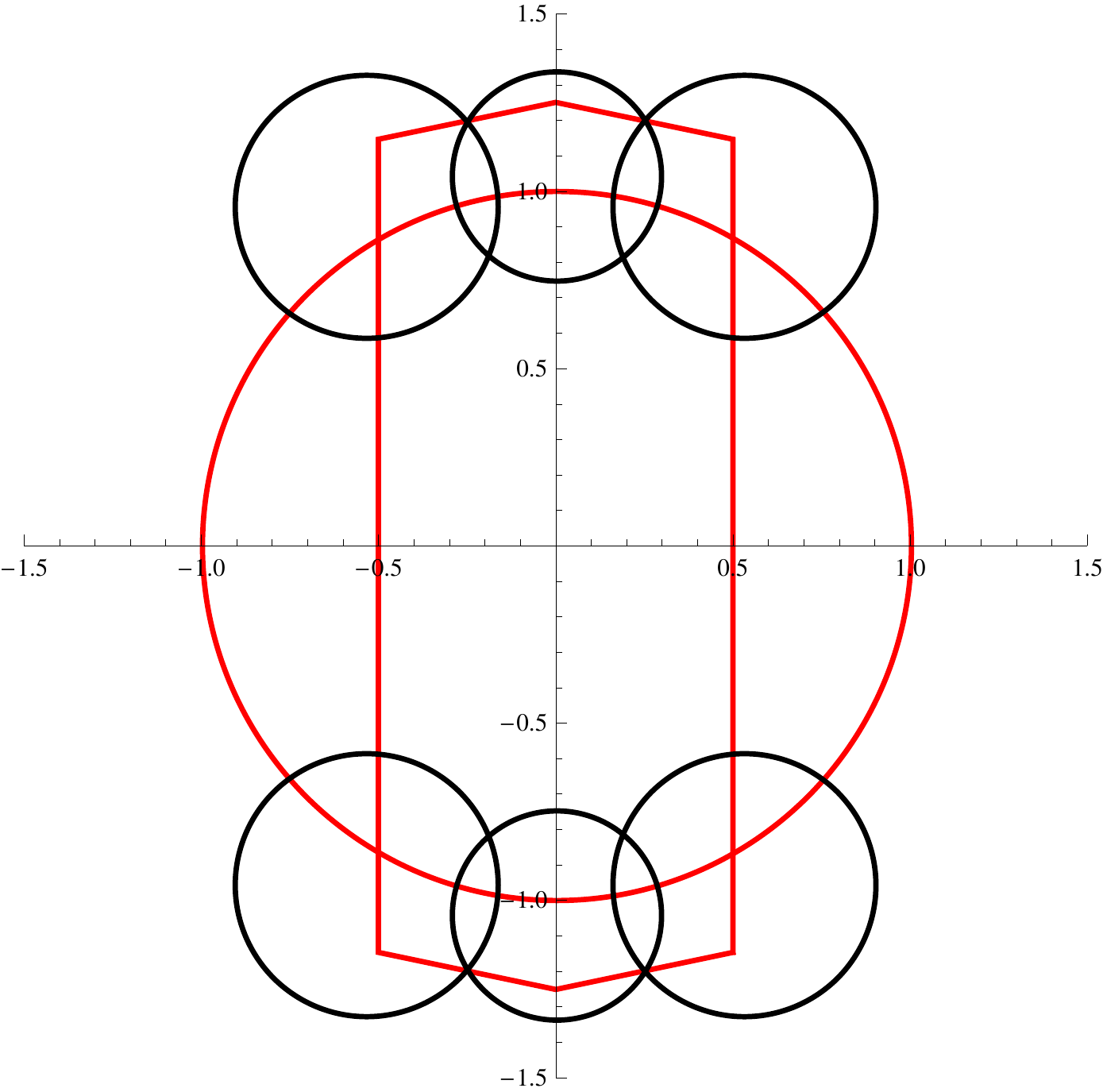}
\label{psl232dim}
}
\subfigure[$\Gamma_{(23)}$]{
\includegraphics[scale=0.3]{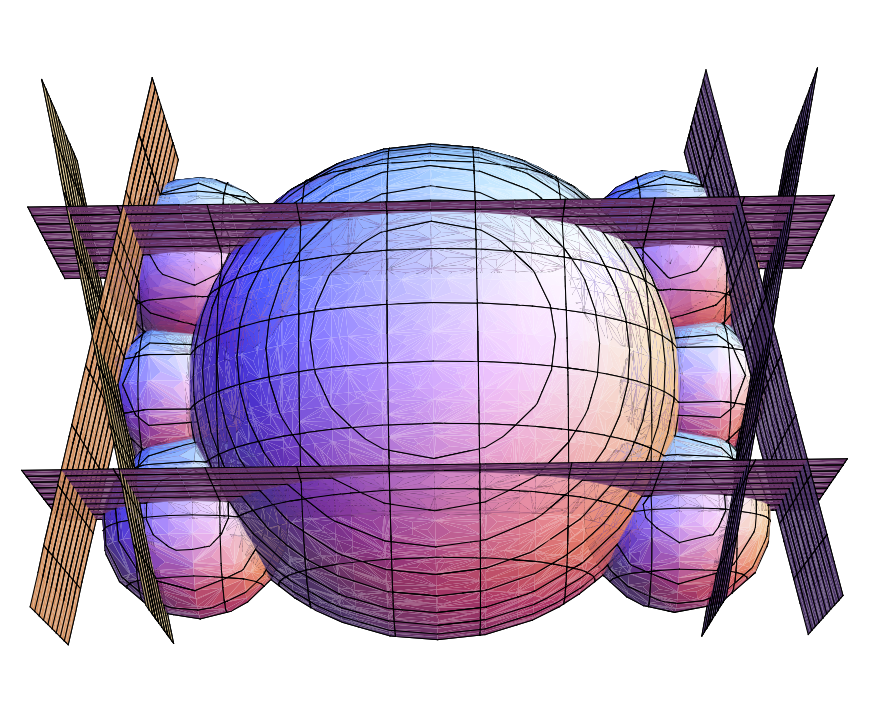}
\label{psl23}
}
\caption{}
\end{figure}

\begin{figure}[H]
\centering
\subfigure[$\Gamma_{(27)}$]{
\includegraphics[scale=0.2]{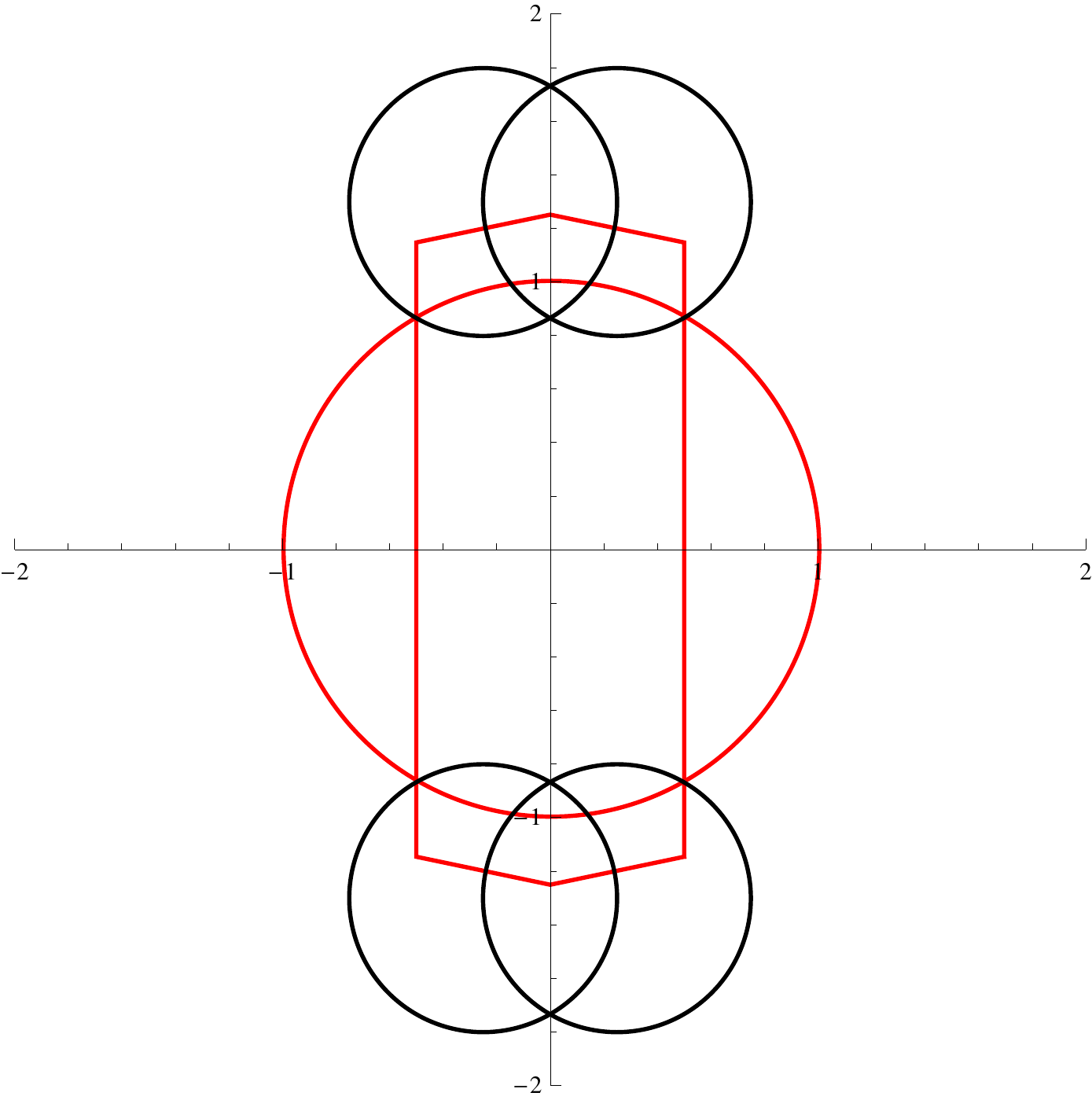}
\label{psl272dim}
}
\subfigure[$\Gamma_{(27)}$]{
\includegraphics[scale=0.3]{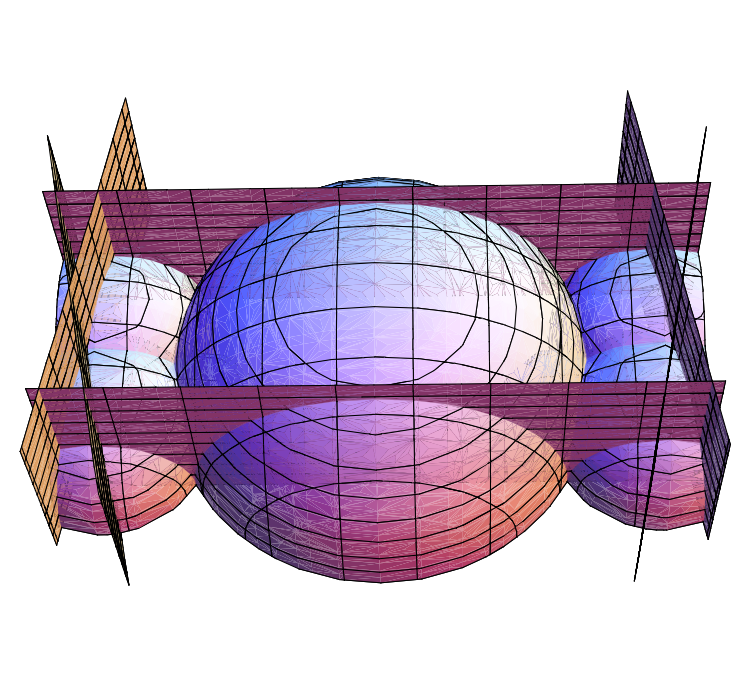}
\label{psl27}
}
\caption{}
\end{figure}

\section{Group Rings}
\label{secgr}

Let $G$ be a finite group. As an application of our algorithm given in Section~\ref{secappl} we are able to find finitely many generators in $\U (\Z G)$ that generate a subgroup of finite index in orders of simple components of $\Q G$ that are  non-commutative division algebras, which are quaternion algebras over $\Q$ or quadratic imaginary extensions of $\Q$, as well as for $2$-by-$2$ matrix algebras over $\Q$ or a quadratic imaginary extension of $\Q$.
This together with the results mentioned in the introduction allow us to describe finitely many generators for a subgroup of finite index in $\U (\Z G)$ for many more groups than previously known. For simplicity we state our result in the case $G$ is nilpotent.

If all Wedderburn components of $\ \Q G$ contain an order whose group of reduced norm one units acts  discretely  on $\HQ^3$, then the finite group $G$ is said to be  of Kleinian type. This subject was treated in \cite{olt-drio} and the classification of these groups was completed in  \cite{jespers-pita-rio-ruiz}. An interesting consequence of this classification  is that $G$ is of Kleininan type if and only if  $\U (\Z G)$ is virtually a direct product of free-by-free groups.

So suppose the finite group $G$ is nilpotent.
For a subset $X$ of $G$ we denote by $\widetilde{X}$ the element $\frac{1}{|X|}\sum_{x\in X}x \in \Q G$. If $X$ is a subgroup then $\widetilde{X}$
is an idempotent of $\Q G$. If, furthermore, $X$ is a normal subgroup then $\widetilde{X}$ is a central idempotent of $\Q G$.
In \cite{jes-olt-rio} the rational representations of a finite nilpotent group $G$
have been explicitly described by exhibiting a set of matrix units of each simple component of
$\Q G$. In particular, a precise  description is given when
a simple component $\Q Ge$ ($e$ a primitive central idempotent) is of exceptional type (this also follows from \cite{jes-leal-deg12} or
\cite{baniq}). \textit{These simple
components  are  one of the following algebras: $\h (\Q( \xi_{2^{m-1}} + \xi_{2^{m-1}}^{-1},\xi_{n}))$ with $1\neq n$ a positive odd integer such that $2$ has odd order modulo $n$ and an integer $m \geq 3$, $M_2(\Q)$,
$M_2(\Q (\sqrt{-2}))$, $M_2(\Q (i))$, $M_2(\Q (\sqrt{-3}))$ or $M_2(\h(\Q))$. Let $G_{e}$ denote the stabilizer of $e$ in $G$. Note that $\Q Ge $ is a simple component of $\Q G \widetilde{G}\cong \Q (G/G_{e})$. For each simple algebra
$\Q Ge$, a  description of $G/G_e$ is  given and, in case $\Q G e$ is a matrix algebra, a complete set of matrix units  is also given}
(see also \cite{jes-leal-deg12}).
The exceptional simple components $\Q Ge$ of the type $\h(\Q (\xi_{2^{m-1}} + \xi_{2^{m-1}}^{-1},\xi_{n}))$ are determined by groups  $G/G_{e}$ of the type
$\Q_{2^{m}}\times C_{n}$ with $1\neq n$ a positive integer such that $2$ has odd order modulo $n$ and an integer $m\geq 3$. The exceptional simple components of the type
$M_{2}(\h (\Q ))$ are determined by some specific $2$-groups of order at most $64$.
The exceptional simple components $\Q Ge$ of the type $M_{2}(F)$ are determined by groups  $G/G_{e}$ of the following type:
$D_{8}$ the dihedral group of order $8$, $D_{16}^{-}=\langle a,b\mid a^{8}=b^{2}=1, ba =a^{3}b\rangle$ the semidihedral group of order $16$, $D_{16}^{+}=\langle a,b \mid a^{8}=b^{2}=1,
ba =a^{5}b\rangle$  the modular group of order 16, ${\cal D} =\langle a,b,c \mid a^{2}=b^{2}=c^{4}=1,ac=ca,bc=cb, ba=c^{2}ab\rangle$, $D_{8}\times C_{3}$, $Q_{8}\times C_{3}$, $D^{+}= \langle a,b,c \mid a^{4}=1, b^{2}=1, c^{4}=1, ca=ac, cb=bc,
            ba=ca^{3}b\rangle $.
In all of these cases, in \cite{jes-olt-rio}, the idempotent $e$ has been explicitly described as well as an  explicit
set of matrix units $E_{11}, E_{22}, E_{12}, E_{21}$ is given. It is this set of units that we will use to describe the following congruence subgroup of level $m$ in $\Q Ge$
(with $m$ a positive integer):
\begin{equation*}
\Gamma_{2,N}(m\O )=\left(1+\sum_{1\leq i,j\leq 2}(m\O )E_{ij} \right) \cap \SL_2(\O) =(1+M_2(m\O ) ) \cap  \SL_2(\O).
\end{equation*}
With the DAFC algorithm one can calculate a finite set of generators for such a group. If one chooses the integer $m$ appropriately then
$1-e+\Gamma_{2,N}(m\O)\subseteq \Z G$. Generators of these groups are  the  units that are used in part 3 of the following result.
The units listed in part 4 are to deal with orders determined by fixed point free groups. Because we exclude simple components that are
division algebras that are not totally definite quaternion algebras the only fixed non-commutative point free epimorphic images of $G$  that can
occur are $Q_{8}\times C_{n}$ with the order of $2$ modulo $n$ even. We then use the matrix idempotents (as part of a set of matrix units) determined
in \cite{jes-olt-rio}. Because of the results mentioned in the introduction, the units described in part  4 of the theorem generate a subgroup
of finite index in the reduced normed one  units of the respective components. The multiple $m$ guarantees that they belong to $\Z G$.
So, all the above together with the results stated in the introduction give us the following result.

\begin{theorem}\label{generators-nilpotent-group} Let $G$ be a nilpotent finite group
of nilpotency class $n$. Assume that the rational
group algebra $\Q G$ does not have simple components of the type
$\h (\Q (\xi_{n},\xi_{2^{m-1}}+\xi_{2^{m-1}}^{-1}))$ with $n$ an odd integer so that the order of $2$ mod $n$ is odd
and an integer $m\geq 3$, or $M_{2}(\h(\Q))$
(equivalently $G$ does not have epimorphic images of the type $Q_{8}\times C_n$ for such $n$ nor some special $2$-groups of order at most $64$.)
Then the group
generated by the following units is of finite
index in $\U (\Z G)$:
\begin{enumerate}[(i)]
\item  $b_{(n)}$, with $b$ a Bass cyclic unit in $\Z G$,
\item the
bicylic units in $\Z G$,
\item generators for the groups $\Gamma_{2,N}(m\O)$  (where for the respective congruence groups we use the matrix units
  described in \cite{jes-olt-rio})  with
$N$ a normal subgroup of $G$ so that
\begin{enumerate}
\item $\O =\Z$, $m=8|N|$ and  $G/N =D_{8}$,
\item $\O =\Z [sqrt{-2}]$, $m=8|N|$ and  $G/N=D_{16}^{-}$,
\item $\O =\Z [i]$, $m=8|N|$ and  $G/N=D_{16}^{+}$,
\item $\O =\Z [i]$, $m=2|N|$ and $G/N={\cal D}$,
\item $\O = \Z [\sqrt{-3}]$, $m=24|N|$ and  $G/N =D_{8}\times C_{3}$,
\item $\O =\Z [\sqrt{-3}]$, $m=24|N|$ and  $G/N =Q_{8}\times C_{3}$,
\item $\O =\Z [i]$, $m=32|N|$ and  $G/N =D^{+}$,
\end{enumerate}
\item  $u_{g,N}=1+mE_{11} g E_{22}$ and $u'_{g,N}=1+mE_{22}gE_{11}$, with $m=|N|2n$, $g\in G$ and $N$ a normal subgroup of $G$ so that $G/N=Q_{8}\times C_{n}=\langle a,b,c \mid
a^{4}=1,a^{2}=b^{2},ba=a^{-1}b,ca=ac,
cb=bc\rangle$ such that $2$ has even  order in $\U (\Z_{n})$, where $E_{11}=e\frac{1}{2}(1+xa+yab)$ and $E_{22}=e\frac{1}{2}(1-xa-yab)$. Moreover $e=\widetilde{N}
\frac{1}{2}(1-a^{2})\left(
1-\frac{1}{n}(\hat{c})\right)$, where $x=\frac{1}{2}(\alpha +b\alpha b^{-1}),\ y=b^{3}\frac{1}{2}(\alpha -b\alpha b^{-1})\in \Z \langle c^{n/p}\rangle $, and $\alpha =\prod_{k=0}^{m-1} \left( 1+bc^{(n/p)2^{k}} \right)$
and $p$ is a prime divisor of $n$  so that $2^{p}\equiv -1 \mod p$.
\end{enumerate}
\end{theorem}

In \cite{gia-seh} different units were
used for (iv). In general, the result fails if one
does not include the units listed in (iii) (we refer the reader to \cite[Section 25]{seh2}).

The groups listed in Theorem~\ref{generators-nilpotent-group}(iii) (with for example $|N|=1$) are of finite covolume (coarea)  but not cocompact and a set of generators can be calculated  using the DAFC. However, for most of them it gives a too large set of generators to be listed here. As a matter of example, we therefore restrict ourselves to determine a set of generators for a subgroup of finite index for each of the following groups:     $\Gamma_{2}(8\Z)$,
$\Gamma_{2}(2\Z[\sqrt{-2}])$ and $\Gamma_{2}(2\Z[i])$.

Let $n_0\in \N$. In $\PSL_2(\Z)$ consider the discrete Fuchsian  subgroup
\begin{equation*}
\Gamma_{2}(n_0\Z)= \lbrace \gamma =
\begin{pmatrix}
1+n_0a & n_0b\\
n_0c & 1+n_0d
\end{pmatrix}, a,b,c,d \in \Z \rbrace
\end{equation*}

The following lemma defines the group $\Gamma_{2}(n_0\Z)$ in terms of a system of algebraic equations.
\begin{lemma}\label{eqcong8}
Let $\gamma \in \Gamma_{2}(n_0\Z)$. Write $\gamma =1+n_0\hat{\gamma}$, where $\hat{\gamma}=
\begin{pmatrix}
a & b\\
c & d
\end{pmatrix}$. If we set $\|\hat{\gamma}\|^2-2\det (\hat{\gamma})=n$, then we get the following system of equations.
\begin{eqnarray}
 \begin{cases}
\Vert \hat{\gamma} \Vert ^2-2\det (\hat{\gamma})=n \\
\tr(\hat{\gamma})+n_0\det (\hat{\gamma})=0\\
\Vert \gamma \Vert^2= 2+ n_0^2n, n\in  \N
\end{cases}
\end{eqnarray}
\end{lemma}

\begin{proof}
The norm of $\gamma$ is given by $\Vert \gamma \Vert^2=2+n_0^2(\Vert\hat{\gamma}\Vert^2+\frac{2}{n_0}\tr(\hat{\gamma}))$. One easily computes that $\det (\gamma )=1$  if and only if $\ \tr(\hat{\gamma})+n_0\det (\hat{\gamma})=0$. Hence, it follows that $\|\gamma\|^2=2+n_0^2(\|\hat{\gamma}\|^2-2\det (\hat{\gamma}))$. Set $\|\hat{\gamma}\|^2-2\det (\hat{\gamma})=n$ and the result follows.
\end{proof}

Next, working in $\BQ^2$, we look for elements $\gamma$ whose isometric circle passes through $i$. In this case, these are exactly the elements stabilizing $i$. Thus, for such elements $\gamma$,  we have that $(1+n_0a)^2+n_0^2c^2=1$ and therefore  $\gamma = \begin{pmatrix}
1 & n_0b\\
0 & 1
\end{pmatrix}$. Hence, their  defining bisectors, in $\mathbb{H}^2$, which are vertical lines, are given by the equation $Re (n_0bz)+\frac{n_0^2b^2}{2}=0$ or equivalently $x+\frac{n_0b}{2}=0$. The  bisector corresponding to such an element of smallest norm is the line $x+\frac{n_0}{2}=0$.  We use the DAFC to cover the compact region $[\frac{-n_0}{2}\leq x\leq \frac{n_0}{2}]$, up to a finite number of points, which are precisely the remaining ideal vertices of $\Gamma$. For $n_0=8$, we get the compact region $[-4,4]$ which is covered using solutions of the above system with $n\leq 2304$. This is depicted in Figure \ref{cong8}.
As this example gives a lot of generators, we do not write them as matrices as usual, but under a more compact form (the three columns give the entries $a$, $b$, $c$ and $d$ of the $19$ generating matrices).

\begin{theorem}
$\tilde{\Gamma}_2(8\Z ) =\langle g(\gamma)\mid g\in \langle \sigma^2,\phi \rangle, \gamma \in X \rangle < \Gamma_2(8\mathbb{Z})$, where the transformations $\sigma^2$ and $\phi$ are given in Lemma \ref{symmetries} and where $X$ is given by the table below, is of finite index in $\Gamma_2(8\Z)$.

\begin{tabular}{cccc | cccc |cccc}
a & b & c & d & a & b & c & d & a & b & c & d \\
\hline
0 & 0 & 1 & 0 &  -3 & 1 & -9 & 3&  8 & -6 & -11 & 8\\
-1 & 1 & -1 & 1 & 5 & -2 & -8 & 3&  13 & -8 & 8 & -5\\
2 & -1 & 4 & -2 & -10 & 7 & -3 & 2  &  8 & -3 & -22 & 8\\
-7 & 2 & 3 & -1 & 6 & -4 & 9 & -6  &  16 & -26 & -10 & 16\\
-4 & 3 & 5 & -4 & 11 & -4 & 8 & -3&  21 & -8 & 34 & -13\\
7 & -2 & -4 & 1 & -4 & 1 & 15 & -4 & -16 & 6 & 42 & -16\\
 -5 & 2 & 7 & -3 & & & & & & & &
\end{tabular}
\end{theorem}

\begin{figure}[H]
\centering
\includegraphics[scale=0.75]{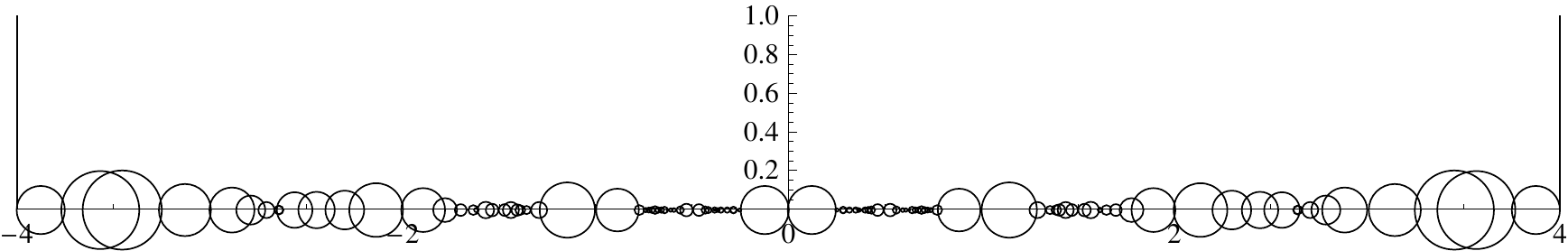}
\caption{Fundamental domain of $\tilde{\Gamma}_2(8\Z )$.}
\label{cong8}
\end{figure}

We now move on to the groups $\Gamma_{2}(2\Z[\sqrt{-2}])$ and $\Gamma_{2}(2\Z[i])$.
The defining system of Diophantine equations to describe the groups (and to obtain generators for these  groups) in the next theorem are exactly the same as in Lemma \ref{eqcong8}, i.e.,

\begin{eqnarray}
 \begin{cases}
\|\hat{\gamma}\|^2-2\det (\hat{\gamma})=n \\
\tr(\hat{\gamma})+n_0\det (\hat{\gamma})=0\\
\| \gamma\|^2= 2+ n_0^2n, n\in  \N
\end{cases}
 \end{eqnarray}
only the entries of $\hat{\gamma}$ are now in $\Z [\sqrt{-d}]$. This is the case for any congruence subgroup which we still have to deal with.

\begin{theorem}
\begin{enumerate}
 \item The subgroup ${\tilde{\Gamma}}_{2}(2\Z [\sqrt{-2}])=\langle -1, g(X_1)\ |\ g\in \langle \sigma^2\circ\tau , \phi \rangle \rangle <\Gamma_{2}(2\Z[\sqrt{-2}])$, where $X_1$ is given below, is of finite index .

\item The subgroup ${\tilde{\Gamma}}_{2}(2\Z [i])=\langle -1,  g(X_2) \ |\ g\in \langle \sigma^2\circ\tau , \phi \rangle \rangle<\Gamma_{2}(2\Z[i])$,where $X_2$ is given below, is of finite index .
\end{enumerate}

\begin{eqnarray*}
X_1 &=& \left\lbrace
\begin{pmatrix}
1 & -2\\
0 & 1
\end{pmatrix},
\begin{pmatrix}
1 & -2\sqrt{-2}\\
0 & 1
\end{pmatrix},
\begin{pmatrix}
-1 & 0\\
2 & -1
\end{pmatrix},
\begin{pmatrix}
-1 & 0\\
-2\sqrt{-2} & -1
\end{pmatrix},
\begin{pmatrix}
-3 & 2\sqrt{-2}\\
-2\sqrt{-2} & -3
\end{pmatrix}, \right. \\
&& \begin{pmatrix}
1-2\sqrt{-2} & 2\sqrt{-2}\\
-2\sqrt{-2} & -1-\sqrt{-2}
\end{pmatrix},
\begin{pmatrix}
-1+2\sqrt{-2} & 2\\
4 & -1-\sqrt{-2}
\end{pmatrix},
\begin{pmatrix}
-1+2\sqrt{-2} & 4\\
2 & -1-\sqrt{-2}
\end{pmatrix},            \\
&&\begin{pmatrix}
-3+2\sqrt{-2} & 4\\
4 & -3-2\sqrt{-2}
\end{pmatrix},
\begin{pmatrix}
-3-2\sqrt{-2} & 2\sqrt{-2}\\
-4\sqrt{-2} & -3+2\sqrt{-2}
\end{pmatrix},\\
&&\left. \begin{pmatrix}
-3-2\sqrt{-2} & 4\sqrt{-2}\\
-2\sqrt{-2} & -3+2\sqrt{-2}
\end{pmatrix},
\begin{pmatrix}
5-2\sqrt{-2} & 4\sqrt{-2}\\
-4\sqrt{-2} & 5+2\sqrt{-2}
\end{pmatrix}
\right\rbrace ,
\end{eqnarray*}

\begin{eqnarray*}
X_2= \left\{
\begin{pmatrix}
1 & -2\\
0 & 1
\end{pmatrix},
\begin{pmatrix}
1 & -2i\\
0 & 1
\end{pmatrix},
\begin{pmatrix}
-1 & 0\\
-2i & -1
\end{pmatrix},
\begin{pmatrix}
-1 & 0\\
2 & -1
\end{pmatrix},
\begin{pmatrix}
-1+2i & 2\\
2 & -1-2i
\end{pmatrix},
\begin{pmatrix}
-1-2i & 2i\\
-2i & -1+2i
\end{pmatrix}
\right\}.
\end{eqnarray*}

\end{theorem}

A fundamental domain and its projection on $\partial \HQ^3$ is given below for the two groups above.

\begin{figure}[H]
\centering
\subfigure[${\tilde{\Gamma}_{2}(2\Z \lbrack \sqrt{-2}\rbrack)}$]{
\includegraphics[scale=0.35]{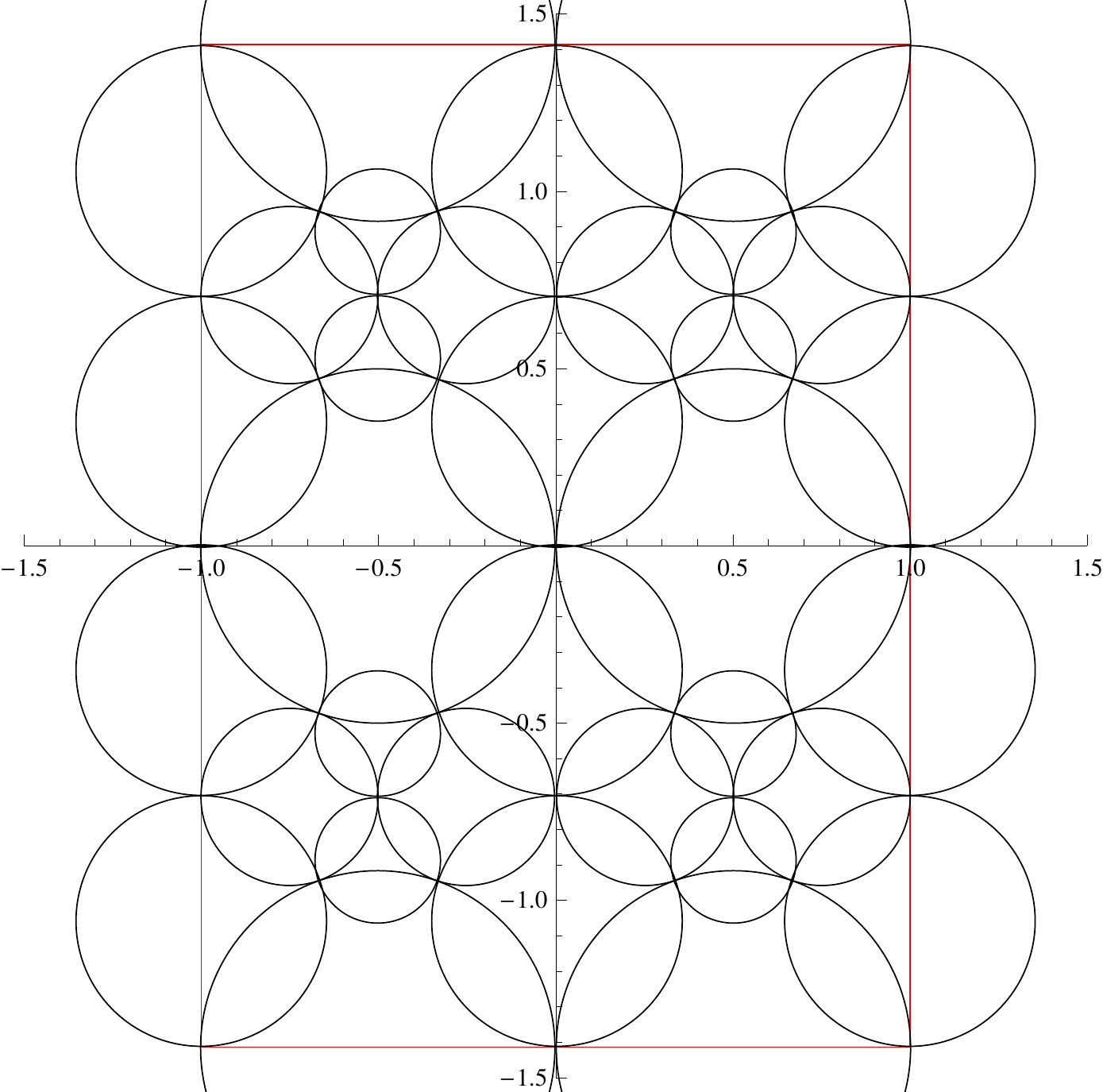}
\label{congruence2zsqrt-22dim}
}
\subfigure[$\tilde{\Gamma}_{2}(2\Z\lbrack\sqrt{-2}\rbrack)$]
{\includegraphics[scale=0.5]{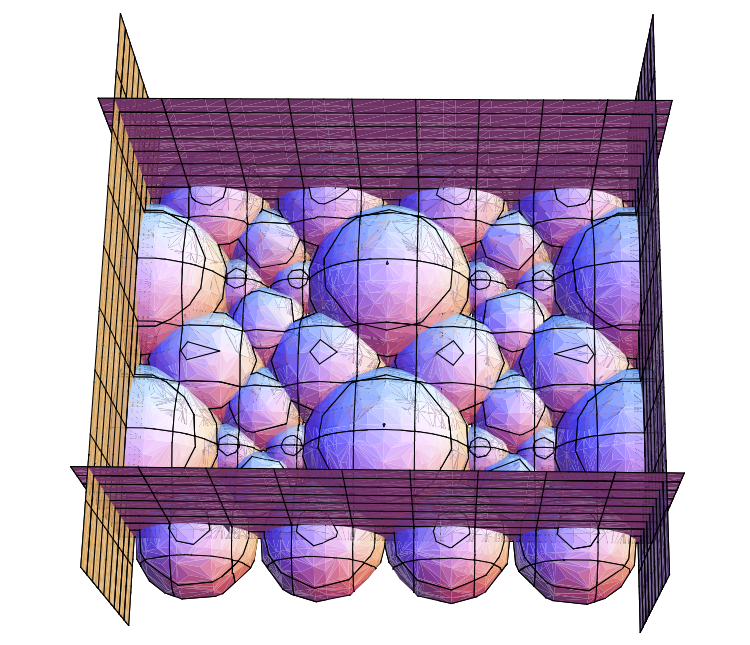}
\label{congruence2zsqrt-2} }
\caption{}
\end{figure}

\begin{figure}[H]
\centering
\subfigure[$\tilde{\Gamma}_{2}(2\Z\lbrack i\rbrack)$]{
\includegraphics[scale=0.35]{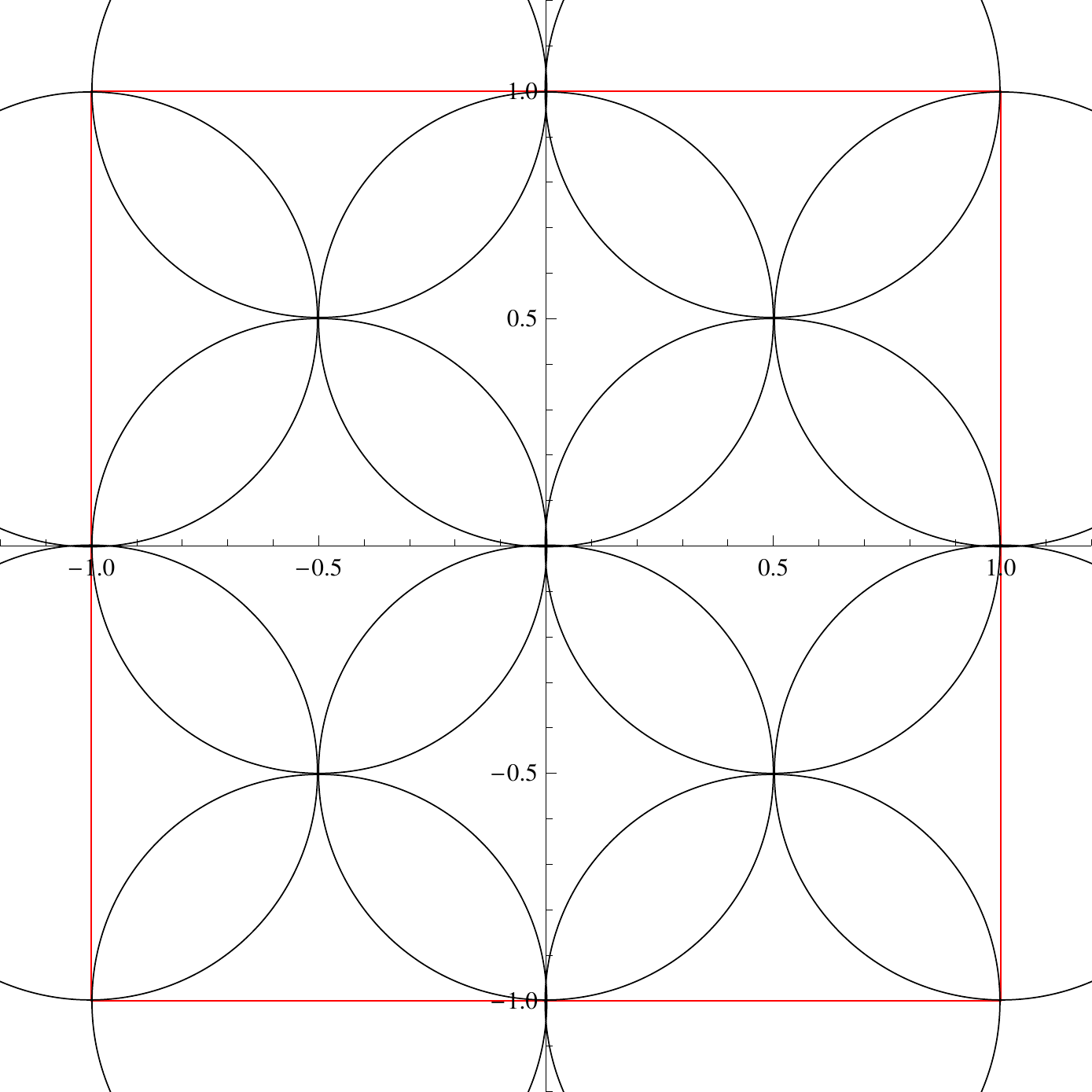}
\label{congruence2zi2dim}
}
\subfigure[$\tilde{\Gamma}_{2}(2\Z \lbrack i\rbrack)$]{
\includegraphics[scale=0.5]{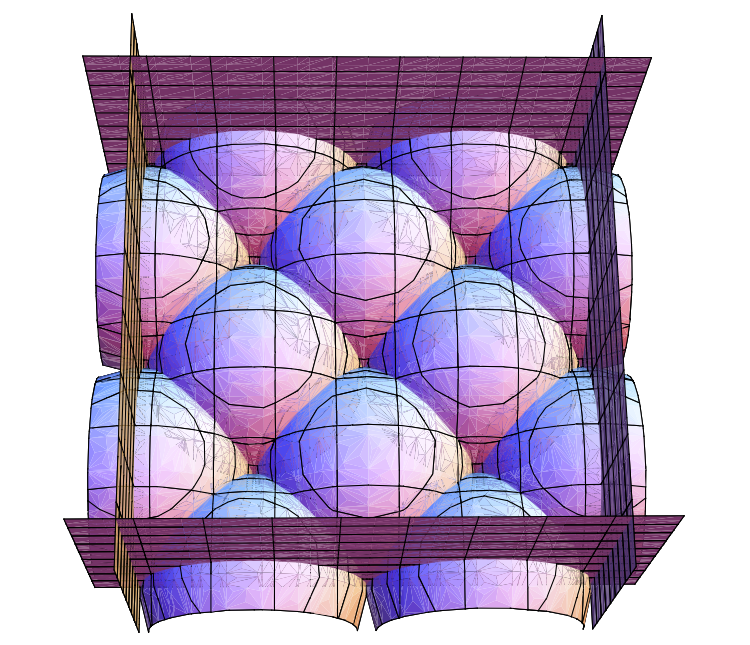}
\label{congruence2zi}
}
\caption{}
\end{figure}

\noindent {\bf{Acknowledgment.}} The second author is grateful to the Vrije Universiteit Brussel for its hospitality during his sabbatical year, while this research was being done. The third author would like to thank Fred Simons for his assistance with Mathematica.


\nocite{*}
\bibliographystyle{abbrv}
\bibliography{KiJeJuSiSBiblio}

\begin{thebibliography}{10}

\bibitem{ami}
S.~A. Amitsur.
\newblock Finite subgroups of division rings.
\newblock {\em Trans. Amer. Math. Soc.}, 80:361--386, 1955.

\bibitem{bak-rehman}
A.~Bak and U.~Rehmann.
\newblock The congruence subgroup and metaplectic problems for {${\rm
  SL}_{n\geq 2}$} of division algebras.
\newblock {\em J. Algebra}, 78(2):475--547, 1982.

\bibitem{baniq}
B.~Banieqbal.
\newblock Classification of finite subgroups of {$2\times 2$} matrices over a
  division algebra of characteristic zero.
\newblock {\em J. Algebra}, 119(2):449--512, 1988.

\bibitem{bas}
H.~Bass.
\newblock The {D}irichlet unit theorem, induced characters, and {W}hitehead
  groups of finite groups.
\newblock {\em Topology}, 4:391--410, 1965.

\bibitem{bass-milnor-serre}
H.~Bass, J.~Milnor, and J.-P. Serre.
\newblock Solution of the congruence subgroup problem for {${\rm
  SL}_{n}\,(n\geq 3)$} and {${\rm Sp}_{2n}\,(n\geq 2)$}.
\newblock {\em Inst. Hautes \'Etudes Sci. Publ. Math.}, (33):59--137, 1967.

\bibitem{beardon}
A.~F. Beardon.
\newblock {\em The geometry of discrete groups}, volume~91 of {\em Graduate
  Texts in Mathematics}.
\newblock Springer-Verlag, New York, 1995.
\newblock Corrected reprint of the 1983 original.

\bibitem{bridson}
M.~R. Bridson and A.~Haefliger.
\newblock {\em Metric spaces of non-positive curvature}, volume 319 of {\em
  Grundlehren der Mathematischen Wissenschaften [Fundamental Principles of
  Mathematical Sciences]}.
\newblock Springer-Verlag, Berlin, 1999.

\bibitem{jesetall}
C.~Corrales, E.~Jespers, G.~Leal, and A.~del R{\'{\i}}o.
\newblock Presentations of the unit group of an order in a non-split quaternion
  algebra.
\newblock {\em Adv. Math.}, 186(2):498--524, 2004.

\bibitem{dooms-jespers-konovalov}
A.~Dooms, E.~Jespers, and A.~Konovalov.
\newblock From {F}arey symbols to generators for subgroups of finite index in
  integral group rings of finite groups.
\newblock {\em J. K-Theory}, 6(2):263--283, 2010.

\bibitem{elstrodt}
J.~Elstrodt, F.~Grunewald, and J.~Mennicke.
\newblock {\em Groups acting on hyperbolic space}.
\newblock Springer Monographs in Mathematics. Springer-Verlag, Berlin, 1998.
\newblock Harmonic analysis and number theory.

\bibitem{floge}
D.~Fl{\"o}ge.
\newblock Zur {S}truktur der {${\rm PSL}_{2}$} \"uber einigen
  imagin\"ar-quadratischen {Z}ahlringen.
\newblock {\em Math. Z.}, 183(2):255--279, 1983.

\bibitem{gia-seh}
A.~Giambruno and S.~K. Sehgal.
\newblock Generators of large subgroups of units of integral group rings of
  nilpotent groups.
\newblock {\em J. Algebra}, 174(1):150--156, 1995.

\bibitem{gromov}
M.~Gromov.
\newblock Hyperbolic groups.
\newblock In {\em Essays in group theory}, volume~8 of {\em Math. Sci. Res.
  Inst. Publ.}, pages 75--263. Springer, New York, 1987.

\bibitem{jeskiefjuretall}
E.~Jespers, S.~O. Juriaans, A.~Kiefer, A.~de~A.~e Silva, and A.~C. Souza~Filho.
\newblock Poincar\'e bisectors in hyperbolic spaces.
\newblock {\em preprint}.

\bibitem{jes-leal-manus}
E.~Jespers and G.~Leal.
\newblock Generators of large subgroups of the unit group of integral group
  rings.
\newblock {\em Manuscripta Math.}, 78(3):303--315, 1993.

\bibitem{jes-leal-deg12}
E.~Jespers and G.~Leal.
\newblock Degree {$1$} and {$2$} representations of nilpotent groups and
  applications to units of group rings.
\newblock {\em Manuscripta Math.}, 86(4):479--498, 1995.

\bibitem{jes-olt-rio}
E.~Jespers, G.~Olteanu, and {\'A}.~del R{\'{\i}}o.
\newblock Rational group algebras of finite groups: from idempotents to units
  of integral group rings.
\newblock {\em Algebr. Represent. Theory}, 15(2):359--377, 2012.

\bibitem{jes-par-seh}
E.~Jespers, M.~M. Parmenter, and S.~K. Sehgal.
\newblock Central units of integral group rings of nilpotent groups.
\newblock {\em Proc. Amer. Math. Soc.}, 124(4):1007--1012, 1996.

\bibitem{jespers-pita-rio-ruiz}
E.~Jespers, A.~Pita, {\'A}.~del R{\'{\i}}o, M.~Ruiz, and P.~Zalesskii.
\newblock Groups of units of integral group rings commensurable with direct
  products of free-by-free groups.
\newblock {\em Adv. Math.}, 212(2):692--722, 2007.

\bibitem{johansson}
S.~Johansson.
\newblock On fundamental domains of arithmetic {F}uchsian groups.
\newblock {\em Math. Comp.}, 69(229):339--349, 2000.

\bibitem{jpp}
S.~O. Juriaans, I.~B.~S. Passi, and D.~Prasad.
\newblock Hyperbolic unit groups.
\newblock {\em Proc. Amer. Math. Soc.}, 133(2):415--423 (electronic), 2005.

\bibitem{jpsf}
S.~O. Juriaans, I.~B.~S. Passi, and A.~C. Souza~Filho.
\newblock Hyperbolic unit groups and quaternion algebras.
\newblock {\em Proc. Indian Acad. Sci. Math. Sci.}, 119(1):9--22, 2009.

\bibitem{jurcal}
S.~O. Juriaans and A.~C. Souza~Filho.
\newblock Free groups in quaternion algebras.
\newblock {\em J. Algebra}, 379:314--321, 2013.

\bibitem{katok}
S.~Katok.
\newblock Reduction theory for {F}uchsian groups.
\newblock {\em Math. Ann.}, 273(3):461--470, 1986.

\bibitem{klei-sur}
E.~Kleinert.
\newblock Units of classical orders: a survey.
\newblock {\em Enseign. Math. (2)}, 40(3-4):205--248, 1994.

\bibitem{KleinertTotDef}
E.~Kleinert.
\newblock Two theorems on units of orders.
\newblock {\em Abh. Math. Sem. Univ. Hamburg}, 70:355--358, 2000.

\bibitem{kle-boo}
E.~Kleinert.
\newblock {\em Units in skew fields}, volume 186 of {\em Progress in
  Mathematics}.
\newblock Birkh\"auser Verlag, Basel, 2000.

\bibitem{Liehl1981}
B.~Liehl.
\newblock On the group {${\rm SL}_{2}$} over orders of arithmetic type.
\newblock {\em J. Reine Angew. Math.}, 323:153--171, 1981.

\bibitem{mac}
C.~Maclachlan and A.~W. Reid.
\newblock {\em The arithmetic of hyperbolic 3-manifolds}, volume 219 of {\em
  Graduate Texts in Mathematics}.
\newblock Springer-Verlag, New York, 2003.

\bibitem{olt-drio}
G.~Olteanu and {\'A}.~del R{\'{\i}}o.
\newblock Group algebras of {K}leinian type and groups of units.
\newblock {\em J. Algebra}, 318(2):856--870, 2007.

\bibitem{passman-perm}
D.~Passman.
\newblock {\em Permutation groups}.
\newblock W. A. Benjamin, Inc., New York-Amsterdam, 1968.

\bibitem{sehmil}
C.~Polcino~Milies and S.~K. Sehgal.
\newblock {\em An introduction to group rings}, volume~1 of {\em Algebras and
  Applications}.
\newblock Kluwer Academic Publishers, Dordrecht, 2002.

\bibitem{ratcliffe}
J.~G. Ratcliffe.
\newblock {\em Foundations of hyperbolic manifolds}, volume 149 of {\em
  Graduate Texts in Mathematics}.
\newblock Springer, New York, second edition, 2006.

\bibitem{rehman}
U.~Rehmann.
\newblock A survey of the congruence subgroup problem.
\newblock In {\em Algebraic {$K$}-theory, {P}art {I} ({O}berwolfach, 1980)},
  volume 966 of {\em Lecture Notes in Math.}, pages 197--207. Springer, Berlin,
  1982.

\bibitem{riley}
R.~Riley.
\newblock Applications of a computer implementation of {P}oincar\'e's theorem
  on fundamental polyhedra.
\newblock {\em Math. Comp.}, 40(162):607--632, 1983.

\bibitem{rit-seh2}
J.~Ritter and S.~K. Sehgal.
\newblock Construction of units in group rings of monomial and symmetric
  groups.
\newblock {\em J. Algebra}, 142(2):511--526, 1991.

\bibitem{rit-seh1}
J.~Ritter and S.~K. Sehgal.
\newblock Construction of units in integral group rings of finite nilpotent
  groups.
\newblock {\em Trans. Amer. Math. Soc.}, 324(2):603--621, 1991.

\bibitem{rit-seh3}
J.~Ritter and S.~K. Sehgal.
\newblock Units of group rings of solvable and {F}robenius groups over large
  rings of cyclotomic integers.
\newblock {\em J. Algebra}, 158(1):116--129, 1993.

\bibitem{seh2}
S.~K. Sehgal.
\newblock {\em Units in integral group rings}, volume~69 of {\em Pitman
  Monographs and Surveys in Pure and Applied Mathematics}.
\newblock Longman Scientific \& Technical, Harlow, 1993.
\newblock With an appendix by Al Weiss.

\bibitem{seh3}
S.~K. Sehgal.
\newblock Group rings.
\newblock In {\em Handbook of algebra, {V}ol. 3}, pages 455--541.
  North-Holland, Amsterdam, 2003.

\bibitem{shirvani}
M.~Shirvani and B.~A.~F. Wehrfritz.
\newblock {\em Skew linear groups}, volume 118 of {\em London Mathematical
  Society Lecture Note Series}.
\newblock Cambridge University Press, Cambridge, 1986.

\bibitem{swan}
R.~G. Swan.
\newblock Generators and relations for certain special linear groups.
\newblock {\em Advances in Math.}, 6:1--77 (1971), 1971.

\bibitem{vaserstein}
L.~N. Vaser{\v{s}}te{\u\i}n.
\newblock Structure of the classical arithmetic groups of rank greater than
  {$1$}.
\newblock {\em Mat. Sb. (N.S.)}, 91(133):445--470, 472, 1973.

\bibitem{Venkataramana1994}
T.~Venkataramana.
\newblock On systems of generators of arithmetic subgroups of higher rank
  groups.
\newblock {\em Pacific J. Math.}, 166(1):193--212, 1994.

\end{thebibliography}

\vspace{.25cm}

\noindent Department of Mathematics, \newline Vrije
Universiteit Brussel,\newline Pleinlaan 2, 1050
Brussel, Belgium\newline
emails: efjesper@vub.ac.be and akiefer@vub.ac.be

\vspace{.25cm}

\noindent Instituto de Matem\'atica e Estat\'\i
stica,\newline Universidade de S\~ao Paulo
(IME-USP),\newline Caixa Postal 66281, S\~ao
Paulo,\newline CEP  05315-970 - Brasil \newline
email: ostanley@usp.br

\vspace{.25cm}

\noindent Departamento de Matematica\\
Universidade Federal da Paraiba\\
e-mail: andrade@mat.ufpb.br

\vspace{.25cm}

\noindent Escola de Artes, Ci\^encias e
Humanidades,\newline Universidade de S\~ao Paulo
(EACH-USP),\newline Rua Arlindo B\'ettio, 1000,
Ermelindo Matarazzo, S\~ao Paulo, \newline CEP
03828-000 - Brasil \newline email:
acsouzafilho@usp.br

\end{document}